\documentclass[11pt]{amsart}
\usepackage{amsmath, amssymb, amsthm, verbatim, import}
\usepackage[scaled]{helvet} 
\usepackage[T1]{fontenc}
\usepackage{textcomp}
\normalfont
\usepackage{mathrsfs}

\usepackage[usenames,dvipsnames]{xcolor}

\usepackage[margin=1in]{geometry}
\usepackage[color, final]{showkeys} 		
\definecolor{refkey}{rgb}{0.8,0.8,0.8}
\definecolor{labelkey}{rgb}{0.9,0,0.1}

\usepackage{graphicx}
\usepackage[utf8]{inputenc}
\usepackage[english]{babel}
\usepackage{amssymb}
\usepackage{tikz-cd}
\usepackage{csquotes}

\usepackage{hyperref}
\usepackage[capitalize]{cleveref}
\crefname{ineq}{Ineq.}{inequalities}
\creflabelformat{ineq}{#2{\upshape(#1)}#3} 
\usepackage[normalem]{ulem}
\usepackage{enumerate}
\usepackage{enumitem}
\usepackage{todonotes}
\usepackage{xspace}
\usepackage{tcolorbox}

\hypersetup{
  colorlinks   = true, 
  urlcolor     = blue, 
  linkcolor    = blue, 
  citecolor   = red 
}

\makeatletter
\newtheorem*{rep@theorem}{\rep@title}
\newcommand{\newreptheorem}[2]{%
	\newenvironment{rep#1}[1]{%
		\def\rep@title{#2 \ref{##1}}%
		\begin{rep@theorem}}%
		{\end{rep@theorem}}}
\makeatother

\newtheorem{theorem}{Theorem}[section]
\newreptheorem{theorem}{Theorem}
\newtheorem{lemma}[theorem]{Lemma}
\newtheorem{proposition}[theorem]{Proposition}
\newtheorem{corollary}[theorem]{Corollary}
\newreptheorem{corollary}{Corollary}

\newtheorem*{claim*}{Claim}

\theoremstyle{definition}
\newtheorem{definition}[theorem]{Definition}

\newtheorem{remark}[theorem]{Remark}

\theoremstyle{remark}

\numberwithin{equation}{section}

\newcommand{\mc}{\mathcal}

\newcommand{\Isom}{\mathrm{Isom}}

\newcommand{\Z}{\mathbb{Z}}

\newcommand{\Q}{\mathbb{Q}}
\newcommand{\R}{\mathbb{R}}

\newcommand{\N}{\mathbb{N}}

\DeclareMathOperator{\Hom}{Hom}

\DeclareMathOperator{\id}{id}
\DeclareMathOperator{\Homeo}{Homeo}

\DeclareMathOperator{\diam}{diam}
\DeclareMathOperator{\rank}{rank}

\renewcommand{\bar}{\overline}

\newcommand{\cout}[1]{}

\definecolor{darkcyan}{rgb}{0. 0.65, 0.65}

\newcommand{\eps}{\epsilon}
\newcommand{\minus}{-}

\def\bt{\begin{theorem}}
\def\et{\end{theorem}}
\def\bd{\begin{definition}}
\def\ed{\end{definition}}
\def\bl{\begin{lemma}}
\def\el{\end{lemma}}

\def\be#1\ee{\begin{align}\begin{split} #1 \end{split}\end{align}}
\def\beq#1\eeq{\begin{align*}\begin{split} #1 \end{split}\end{align*}}

\newcommand{\taumod}{\tau_{\mathrm{mod}}}
\newcommand{\sigmamod}{\sigma_{\mathrm{mod}}}
\newcommand{\amod}{a_{\mathrm{mod}}}
\newcommand{\tamod}{\tilde{a}_{\mathrm{mod}}}
\newcommand{\ost}{\mathrm{ost}}
\newcommand{\st}{\mathrm{st}}
\newcommand{\flags}{M}
\newcommand{\vbdry}{\partial_{\infty}}
\newcommand{\dt}{\mathrm{d}_{\mathrm{T}}}

\newcommand{\coding}{\mathbf}

\DeclareMathOperator{\PGL}{PGL}

\definecolor{grey}{rgb}{0.22, 0.32, 0.51}

\begin{document}

\title{Stability for boundary actions of cocompact lattices in
  Euclidean buildings}

\author{Thang Nguyen}
\address{Department of Mathematics, Florida State University, 
    Tallahassee, FL, 32304}
\email{tqn22@fsu.edu}

\author{Theodore Weisman}
\address{Department of Mathematics, University of Michigan, 
    Ann Arbor, MI, 48104}
\email{tjwei@umich.edu}

\thanks{T. N. was supported in part by Simons Travel Support for
  Mathematicians grant MPS-TSM-00002547.}

\thanks{T. W. was supported in part by NSF grant DMS-2202770.}

\date{\today}
\begin{abstract}
  When $X$ is a locally compact Euclidean building, the isometry group
  of $X$ acts by homeomorphisms on the space of $k$-simplices in the
  visual boundary of $X$. We consider perturbations of these actions
  for discrete groups of isometries acting with compact quotient on
  $X$, showing that all small enough perturbations are semi-conjugate
  to the original action. This proves in particular that, when $Q$ is
  any parabolic subgroup in a semisimple $p$-adic Lie group $G$, the
  induced action of a cocompact lattice in $G$ has a topologically
  stable action on $G/Q$.
\end{abstract}
\maketitle

\tableofcontents

\section{Introduction}

Zimmer's program \cite{Zimmer87} seeks to understand and classify actions of large groups, particularly higher-rank lattices, on manifolds. While substantial progress has been achieved for lattices in real Lie groups (for examples \cite{Sullivan85,Ghys93, Kanai96, KatokSpatzier97,FarbShalen00, FisherMargulis05,BrownFisherHurtado16,BrownFisherHurtado20,Pecastaing20, BrownFisherHurtado21, DamjanovicSpatzierVinhageXu22,KapovichKimLee19, ConnellIslamNguyenSpatzier23,Pecastaing24, AFZ24}), much less is known in the $p$-adic setting. One of the fundamental obstacles is the absence of a satisfactory theory of smooth dynamics for totally disconnected groups. In the real case, many rigidity arguments rely on differentiability and on the structure theory of Lie groups, whose foundations originate in the solution of Hilbert's Fifth Problem. The $p$-adic setting naturally leads to analogous questions concerning actions of totally disconnected locally compact groups on manifolds. Closely related is the Hilbert--Smith Conjecture, which predicts that if a locally compact group acts faithfully on a connected manifold, then it must be a Lie group; equivalently, the $p$-adic integers $\Z_p$ cannot act faithfully on a manifold. Despite significant efforts, there has been relatively little progress toward a $p$-adic analogue of Zimmer's program for actions on manifolds.

A natural approach is to replace manifolds by geometric spaces that play the role of homogeneous manifolds in the $p$-adic world. For semisimple $p$-adic groups (a semisimple Lie group over a non-Archimedean local field), flag varieties and their associated boundaries provide such a class of spaces. These spaces carry rich geometric and combinatorial structures, arise naturally from the theory of algebraic groups and buildings, and serve as $p$-adic counterparts of flag manifolds and homogeneous spaces in the real setting. In this paper, we study actions of lattices and other large subgroups on $p$-adic flag spaces and investigate the rigidity phenomena that arise. Our goal is to understand to what extent the strong rigidity properties predicted by Zimmer's program persist in this non-Archimedean setting and to develop new techniques adapted to the geometry of $p$-adic groups and their boundaries.

We work with $p$-adic flag spaces using their connection to \emph{spherical} and \emph{Euclidean buildings}, and heavily exploit the CAT($k$) geometry of these buildings. This links our result to the general theory of ``actions at infinity'' for groups of isometries of a ``nonpositively curved'' metric space $X$---for example, the action of a Gromov-hyperbolic group on its Gromov boundary, or a CAT(0) group on its visual boundary. Since such boundary actions often encode detailed features of the original action on $X$, it is natural to ask what features of some boundary action $\Gamma \to \Homeo(\partial X)$ are perturbed in the space $\Hom(\Gamma, \Homeo(\partial X))$. Questions of this form have a long history,
and some of their answers form an important part of rigidity theory, dating back to
\cite{Sullivan85, Ghys93, Kanai96, KatokSpatzier97}. They are further motivated by the crucial role that induced boundary maps play in the celebrated rigidity results of Mostow \cite{Mostow68} and Margulis \cite{Margulis69}.

For the statement of our main result, let $X$ be a \emph{locally
  compact Euclidean building}, i.e. a CAT(0) space locally modeled on
a Euclidean Coxeter complex $(\amod, W)$; see \cite{KleinerLeeb97} for
background. The visual boundary of this space has the structure of a \emph{spherical building}, meaning it decomposes as a union
of simplices, each canonically identified with a \emph{model simplex
  $\sigmamod$}. Each face $\taumod \subseteq \sigmamod$ therefore
defines a \emph{space of simplices} in $\partial_\infty X$, giving a
finite family of different bordifications of $X$. The corresponding
boundaries are called the \emph{flag boundaries} for the Euclidean
building $X$. Each flag boundary $M$ is a compact space homeomorphic
to a Cantor set, and any subgroup $\Gamma < \Isom(X)$ acts by
homeomorphisms on $M$. We consider perturbations of this $\Gamma$-action,
and prove the following theorem:
\begin{theorem}
  \label{thm:main_theorem}
  Let $X$ be a locally compact Euclidean building, let $\Gamma$ be a
  group acting properly and cocompactly by isometries on $X$, and let
  $\rho_0:\Gamma \to \Homeo(\flags)$ denote the induced action of
  $\Gamma$ on the space $M$ of $\taumod$-flags. Then $\rho_0$ is
  \emph{topologically stable}, i.e. there is a neighborhood $U$ of
  $\rho_0$ in $\Hom(\Gamma, \Homeo(\flags))$ consisting of actions
  semi-conjugate to $\rho_0$.
\end{theorem}

Here we say that two actions $\rho, \rho':\Gamma \to \Homeo(M)$ are
\emph{semi-conjugate} if there exists a $(\rho, \rho')$-equivariant
surjection $M \to M$. Semi-conjugacy is the appropriate notion of
equivalence when considering $C^0$ perturbations of an action
$\Gamma \to \Homeo(M)$; thus, the theorem above can also be viewed
as a local rigidity result for the ``standard'' $\Gamma$-action
$\rho_0$ on $M$. 
\begin{remark}
\label{rem:non_conjugacy}
Semi-conjugacy is the optimal conclusion in the setting of
\Cref{thm:main_theorem}: in \Cref{sec:blowsupexample} we give an
example showing that that there can exist arbitrarily small perturbations of $\rho_0$ that are \emph{not} conjugate to $\rho$. On the other hand, if one asks for the perturbed actions to be bi-Lipschitz, then, by the same argument as in \cite[Section 4.4]{ConnellIslamNguyenSpatzier23}, the semi-conjugacy can be upgraded to a conjugacy.
\end{remark}

When $G$ is a semisimple Lie group over a non-Archimedean local field, then $G$ acts properly
discontinuously by isometries on a locally compact Euclidean building
$X$ as above, and the different flag boundaries of $X$ are canonically
identified with the quotient spaces $G/Q$ for the various parabolic
subgroups $Q < G$. So \Cref{thm:main_theorem} implies:
\begin{corollary}
  \label{cor:p-adic_stability}
  Let $G$ be a semisimple Lie group over a non-Archimedean local field, let $Q < G$ be a parabolic
  subgroup, and let $\Gamma$ be a lattice in $G$. The
  induced action $\rho_0:\Gamma \to \Homeo(G/Q)$ is topologically
  stable.
\end{corollary}

Note that a lattice in a $p$-adic Lie group is necessarily uniform (i.e., cocompact). Also, if every irreducible factor of a Euclidean building has rank at least $3$, then, by a theorem of Tits \cite{Tits74}, the building is the Bruhat--Tits building associated with a $p$-adic Lie group. Consequently, in this setting, \Cref{thm:main_theorem} is equivalent to \Cref{cor:p-adic_stability}.

\subsection{Analogous results}

\Cref{thm:main_theorem} echoes several recent results in the setting
of both (coarse) negative curvature and (smooth) nonpositive
curvature. For example, consider the situation where $Q$ is a
parabolic subgroup in a semisimple real Lie group $G$, and $\Gamma$ is
a uniform lattice in $G$; then the quotient $G/Q$ is a compact smooth
manifold, and the induced action $\rho_0:\Gamma \to \Homeo(G/Q)$ is
smooth, hence bi-Lipschitz. In \cite{KapovichKimLee19},
Kapovich--Kim--Lee showed that any small perturbation of $\rho_0$ in
the space of bi-Lipschitz actions $\Hom(\Gamma, \mathrm{Lip}(G/Q))$ is
\emph{conjugate} to $\rho_0$. Work of Connell--Islam--Nguyen--Spatzier
\cite{ConnellIslamNguyenSpatzier23} gives the $C^0$ version of this
result, proving that the action $\rho_0$ is also topologically
stable. The work \cite{ConnellIslamNguyenSpatzier23} was inspired by Bowden--Mann \cite{BowdenMann20}, who proved stability in the rank one setting. In the cases of actions with higher regularity, rigidity was obtained by Kanai \cite{Kanai96}, Katok--Spatzier \cite{KatokSpatzier97}, and Brown--Rodriguez Hertz--Wang \cite{AFZ24}. These results can be thought of as analogs of
\Cref{cor:p-adic_stability} for real Lie groups.

In a different setting, work of Bowden--Mann \cite{BowdenMann20},
Mann--Manning \cite{MannManning21}, and Mann--Manning--Weisman
\cite{MannManningWeisman22} proves that topological stability holds
for the induced action of any Gromov-hyperbolic group $\Gamma$ on its
Gromov boundary $\partial \Gamma$. In particular, the result of
Mann--Manning--Weisman implies \Cref{thm:main_theorem} in the case
where the Euclidean building $X$ has rank one (i.e. when $X$ is a tree
with no leaves). Mann--Manning--Weisman also proved a version of
topological stability for boundary actions of \emph{relatively}
hyperbolic groups \cite{MMW24}; this implies in particular that if $G$
is the isometry group of a rank-one symmetric space $X$ of noncompact
type, and $\Gamma$ is a nonuniform lattice in $G$, the induced action
of $\Gamma$ on the visual boundary $\partial X$ is ``relatively
topologically stable.'' No $C^0$ stability result, relative or
otherwise, is currently known for boundary actions of nonuniform
lattices acting on either higher-rank symmetric spaces or higher-rank
Euclidean buildings.

\subsection{Outline of the paper}

Our main theorem relies on a general criterion for topological
stability, which we state and prove in
\Cref{sec:general_stability}. The criterion is based simultaneously on
the notion of a ``meandering-hyperbolic'' group action developed in
\cite{KapovichKimLee19} and the dynamical coding for boundary actions
of arbitrary hyperbolic groups considered in
\cite{MannManningWeisman22}. The rough idea of our criterion is as
follows: we suppose that it is possible to construct a pair of
``locally finite'' combinatorial codings $\mc{G}_1$, $\mc{G}_2$ for
the action by homeomorphisms of some finitely generated group $\Gamma$
on a compact metrizable space $M$. Each coding takes the form of a
finite directed graph, whose vertices are labeled by open subsets of
$M$ and whose edges are labeled by elements of $\Gamma$; if there is
an edge $e$ between two vertices, then the element labeling $e$ takes
one of the vertex sets inside the other. A point $z \in M$ is
``coded'' by $\mc{G}_i$ if one can follow an infinite-length path in
$\mc{G}_i$ to yield an infinite intersection of sets containing
$z$. As in the ``meandering-hyperbolic'' group actions appearing in
\cite{KapovichKimLee19}, we suppose that $\mc{G}_1$ encodes every
point in $M$, and and that $\mc{G}_2$ can be used to ``interpolate''
between any two $\mc{G}_1$-codings of the same point; then we use
arguments inspired by techniques in \cite{KapovichKimLee19},
\cite{MannManningWeisman22} and \cite{MMW24} to show that the codings
can be used to recover the data of the $\Gamma$-action on $M$, and
that they are stable under $C^0$ perturbations.

The rest of the paper is concerned with constructing the codings
$\mc{G}_1$ and $\mc{G}_2$ for the standard actions of a group $\Gamma$
on the flag boundaries of a Euclidean building $X$. In
\Cref{sec:euclidean_building_background} we review necessary
background on the structure of Euclidean buildings and their flag
boundaries, and in \Cref{sec:expansivity,sec:constructing_coders} we
carry out the actual construction. This is the most technical part of
the paper, and involves careful study of the interplay between the
geometry of the building $X$ and the dynamics of the action on its
flag boundaries. A key step is an estimate (proved in
\Cref{sec:expansivity}) which directly relates the expansivity
properties of the action on flags to the metric geometry of the action
on a Euclidean building. Throughout our arguments, we rely heavily on
the \emph{higher-rank Morse lemma} of Kapovich--Leeb--Porti
\cite{KLP2018}, which allows us to assert that certain nice
quasi-geodesic rays in $X$ stay within bounded distance of \emph{Weyl
  cones}---convex subsets of $X$ that play a role similar to that of
geodesic rays in the rank-one setting.

In \Cref{sec:interpolation} we prove that the codings from the
preceding sections satisfy the hypotheses of the general stability
theorem in \Cref{sec:general_stability}, which completes the proof of
\Cref{thm:main_theorem}. The main difficulty here is verifying the
interpolation property mentioned previously, which involves further
application of tools introduced by Kapovich--Leeb--Porti
\cite{KLP2018}. Finally, in \Cref{sec:blowsupexample}, we construct
the counterexample alluded to in \Cref{rem:non_conjugacy}, showing
that the semi-conjugacies in \Cref{thm:main_theorem} cannot in general
be improved to conjugacies.

\subsection{Acknowledgments}

The authors thank the University of Michigan, Institut Henri
Poincar\'e, and LabEx CARMIN (ANR-10-LABX-59-01) for their support and
hospitality, and Katie Mann for helpful discussion.

\section{Stability through automata}
\label{sec:general_stability}

Throughout this section, we let $\Gamma$ be a group and let $M$ be a
compact metrizable space. Our goal is to give a general criterion
which guarantees that some action $\rho:\Gamma \to \Homeo(M)$ is
topologically stable. This requires some initial technical setup.

\subsection{Point coders}

The basic idea behind the stability criterion is to encode the
dynamics of an action $\rho:\Gamma \to \Homeo(M)$ with a finite
combinatorial object. In the following we review one possible formalism
(originally developed in the appendix of \cite{MMW24}) for such a
``coding''.

\begin{definition}
  Let $\rho:\Gamma \to \Homeo(M)$ be an action. A \emph{finitary
    $\rho$-coder} is a finite directed graph
  $\mc{G} = (\mc{V}, \mc{E})$ with vertex labels
  $\{W(v) : v \in \mc{V}\}$ and edge labels
  $\{\alpha(e) : e \in \mc{E}\}$ satisfying the following conditions:
  \begin{enumerate}
  \item Each vertex label $W(v)$ is an open subset of $M$.
  \item Each edge label $\alpha(e)$ is an element of $\Gamma$.
  \item Whenever there is an edge from $z_1$ to $z_2$ labeled by
    $\alpha$, there is an inclusion
    \[
      \overline{\rho(\alpha) W(z_2)} \subset W(z_1).
    \]
  \end{enumerate}
\end{definition}


The paths in a $\rho$-coder $\mc{G}$ ``encode'' points in $M$ as
follows. Suppose that $(e_n)_{n =1}^\infty$ is a sequence of edges
giving an infinite path in a $\rho$-coder $\mc{G}$, and
$(z_n)_{n=1}^\infty$ is the corresponding sequence of terminal
vertices. Then, the third condition above inductively implies that
\[
  \rho(\alpha(e_1) \cdots \alpha(e_n))W(z_n)
\]
forms a decreasing sequence, and that the intersection of all of the
sets in this sequence agrees with
\[
  \bigcap_{n=1}^\infty \rho(\alpha(e_1) \cdots \alpha(e_n))
  \overline{W(z_n)}.
\]
Since $M$ is compact, this intersection is nonempty, and we think of
any point in the intersection as being ``coded'' by this path in
$\mc{G}$. If the intersection happens to be a singleton $\{p\}$, then
it is useful to think of the sequence of group elements
\[
  g_n := \alpha(e_1) \cdots \alpha(e_n)
\]
as ``converging towards'' $p$. Since we want to use codings to study
group actions, it is also useful to allow these codings to ``start''
at any element in $\Gamma$ as they converge towards a point in $M$. We
formalize all of this as follows.
\begin{definition}
  Let $\mc{G}$ be a finitary point coder. A \emph{$\mc{G}$-coding} is
  a pair $\coding{c} = (g_0, (e_n)_{n=1}^\infty)$, where
  $g_0 \in \Gamma$ and $(e_n)_{n=1}^\infty$ is an infinite edge path
  in $\mc{G}$. The sequence of group elements
  \[
    g_k := g_0\alpha(e_1) \cdots \alpha(e_k)
  \]
  is called the \emph{path sequence} associated to $\coding{c}$, and
  the sequence of terminal vertices in the path $(e_n)_{n=1}^\infty$
  is called the \emph{terminal vertex sequence} of $\coding{c}$. The
  group element $g_0$ is called the \emph{initial point} of the
  coding; we also say that $\coding{c}$ \emph{starts at} $g_0$.

  When $p \in M$, we say that a $\mc{G}$-coding $\mathbf{c}$ with
  terminal vertex sequence $(z_n)_{n=1}^\infty$ and path sequence
  $(g_n)_{n=0}^\infty$ is \emph{a $(\mc{G}, \rho)$-coding of $p$} if
  $p$ lies in the intersection
  \begin{equation}
    \label{eq:coding_intersection}
    \bigcap_{n=1}^\infty \rho(g_n)W(z_n).
  \end{equation}
  If the action $\rho$ is understood from context, we might just say
  that $\coding{c}$ is a \emph{$\mc{G}$-coding of $p$}.
\end{definition}

In general, a $\mc{G}$-coding $\coding{c}$ may not be a
$(\mc{G}, \rho)$-coding for a \emph{unique} point $p \in M$. We
introduce additional terminology to describe when this occurs:
\begin{definition}
  \label{defn:contracting_coding}
  If a $\mc{G}$-coding $\coding{c}$ codes a unique point $p \in M$ (in
  other words, if the intersection in \eqref{eq:coding_intersection}
  is a singleton), then we say that $\coding{c}$ is
  \emph{$\rho$-contracting}. We say that a finitary point coder
  $\mc{G}$ is $\rho$-contracting if every $\mc{G}$-coding is
  $\rho$-contracting.
\end{definition}

\begin{remark}
  The original definition of point coders in the appendix of
  \cite{MMW24} requires all point coders to be $\rho$-contracting, but
  it occasionally convenient to relax this condition.
\end{remark}

The following result, proved in \cite{MMW24}, gives a
condition ensuring that distinct codings of a common point (possibly
arising from distinct coders) are ``compatible.''
\begin{lemma}[{\cite[Lemma A.3]{MMW24}}: close codings are compatible]
  \label{lem:uniform_nesting}
  For any finite subset $F \subset \Gamma$ and any $\rho$-contracting
  point coders $\mc{G}, \mc{G'}$, there is a number
  $N = N(\mc{G}, \mc{G'}, F)$ satisfying the following. Suppose that
  $\coding{c}$ is a $\mc{G'}$--coding of $p$ with path sequence
  $(g_k)_{k=0}^\infty$ and terminal vertex sequence
  $(z_k)_{k=1}^\infty$, and $\coding{d}$ is a $\mc{G}$--coding of
  $p$ with path sequence $(h_k)_{k=0}^\infty$ and terminal vertex
  sequence $(y_k)_{k = 1}^\infty$. Then, for any indices
  $m, n \in \N$ satisfying $g_n^{-1}h_m \in F$, we have
  \begin{equation}\label{eq:nestingcond}
    \rho(g_{n+N}) \overline{W(z_{n+N})} \subset \rho(h_{m}) W(y_{m}).
  \end{equation}
\end{lemma}

Motivated by the statement of this lemma, we introduce the following terminology:
\begin{definition}
  Fix $\rho$-coders $\mc{G}$, $\mc{H}$ (not necessarily distinct) and
  a finite subset $F \subset \Gamma$. We say that a $\mc{G}$-coding
  $\coding{c}$ with path sequence $(g_k)_{k=0}^\infty$ and an
  $\mc{H}$-coding $\coding{d}$ with path sequence
  $(h_j)_{j=0}^\infty$ are \emph{$F$-close} if there are arbitrarily
  large indices $m, n$ such that $g_n^{-1}h_m \in F$.
\end{definition}

\begin{remark}
  \label{rem:fclose_indices}
  In the definition above, we really mean that \emph{both} indices $m$
  and $n$ can be taken arbitrarily large, i.e. that for any
  $N \in \N$, there is a choice of indices $m, n$ with
  $g_n^{-1}h_m \in F$ and both $m > N$ and $n > N$. However, if
  $\mc{G}$ and $\mc{H}$ are both $\rho$-contracting and $M$ contains
  no isolated points, this is the same as asking for infinitely many
  distinct pairs of indices $(m, n) \in \N^2$ such that
  $g_n^{-1}h_m \in F$. Indeed, if $m \to \infty$ and $g_n^{-1}h_m \in F$
  for bounded $n$, then it follows that a subsequence of $h_m$ is
  constant and therefore (since the $\rho(h_m)W(y_m)$ are nested)
  $\rho(h_m)W(y_m)$ is eventually constant. If $M$ contains no
  isolated points, this is impossible since the nested intersection
  \[
    \bigcap_{m = 1}^\infty \rho(h_m)W(y_m) = \bigcap_{m = 1}^\infty
    \rho(h_m)\overline{W(y_m)}
  \]
  is a singleton while the left hand side equals $\rho(h_m)W(y_m)$, an open set, for a sufficiently large $m$.
\end{remark}

\subsection{Generating point coders}
\label{sec:generating_coders}

Typically one constructs combinatorial codings for an action
$\rho:\Gamma \to \Homeo(M)$ via some kind of ``expansive data'' for $\rho$. In
this paper, we express this idea through the notion of a
\emph{proto-coder}, which \emph{generates} the codings described
above.

\begin{definition}
  \label{defn:protocoder}
  A \emph{proto-coder} $\mc{S}$ consists of the following data:
  \begin{itemize}
  \item a finite index set $Z$,
  \item a pair of open coverings
    $\{W(z) : z \in Z\}, \{U(z) : z \in Z\}$ indexed by $Z$,
    satisfying $U(z) \subset W(z)$ for each $z \in Z$, and
  \item a choice of group element $\alpha(z) \in \Gamma$ for each
    $z \in Z$.
  \end{itemize}
  We say that the proto-coder $\mc{S}$ is \emph{adapted to $\rho$} if
  it satisfies the following condition: for each pair $y, z \in Z$,
  if the intersection $\rho(\alpha(z))^{-1}U(z) \cap U(y)$ is
  nonempty then we have a proper inclusion
  $\overline{W(y)} \subset \rho(\alpha(z))^{-1}W(z)$.
\end{definition}

\begin{definition}
  \label{defn:stable_system}
  When $\mc{S}$ is a proto-coder adapted to $\rho$, we say that
  $\mc{S}$ is \emph{$\rho$-stable} if, for every pair $y, z \in Z$, we
  have
  \[
    \rho(\alpha(z))^{-1}U(z) \cap U(y) = \emptyset \iff
    \rho(\alpha(z))^{-1}\overline{U(z)} \cap \overline{U(y)} =
    \emptyset.
  \]
  If the action $\rho$ is understood from context, we will just say
  that $\mc{S}$ is \emph{stable}.
\end{definition}

\subsubsection{Constructing point coders from proto-coders}

Any proto-coder $\mc{S}$ adapted to $\rho$ determines a $\rho$-coder
$\mc{G}$.
\begin{definition}
  Let $\mc{S}$ be a proto-coder adapted to $\rho$. The $\rho$-coder
  $\mc{G}(\mc{S}, \rho)$ \emph{generated by $\mc{S}$ and $\rho$} is
  determined as follows:
  \begin{itemize}
  \item The vertex set of $\mc{G}(\mc{S}, \rho)$ is $Z$, the common
    index set of $\mc{S}$.
  \item The vertex labels of $\mc{G}(\mc{S}, \rho)$ are given by the
    sets $W(z)$ for each $z \in Z$.
  \item There is an edge $e$ from $z$ to $y$ in $\mc{G}(\mc{S}, \rho)$
    if and only if the set $\rho(\alpha(z))^{-1}U(z) \cap U(y)$ is
    nonempty. The label of this edge is $\alpha(z)$.
  \end{itemize}
\end{definition}
It is easy to check that this actually defines a finitary
$\rho$-coder. The proof of the following lemma is also
straightforward, and can already be found in the proof of \cite[Lemma
2.10]{MannManningWeisman22}.
\begin{lemma}
  \label{lem:coding_exists}
  Let $\mc{G} = \mc{G}(\mc{S}, \rho)$ be the point coder generated by $\mc{S}$ and $\rho$.
  Then every $p \in M$ has a $\mc{G}$-coding starting from the
  identity element.
\end{lemma}
\begin{proof}
  Fix $p \in M$. We wish to find an infinite edge path
  $(e_n)_{n=1}^\infty$ such that
  \[
    p \in \rho(\alpha(e_1) \cdots \alpha(e_n))W((z_n))
  \]
  for all $n$, where $z_n$ is the terminal vertex of the edge
  $z_n$. Equivalently, we want to find an infinite vertex path
  $(z_n)_{n=1}^\infty$ so that
  \[
    p \in \rho(\alpha(z_1) \cdots \alpha(z_{n-1}))W(z_n)
  \]
  for all $n$. In fact, we will construct a sequence so that
  \begin{equation}
    \label{eq:code_point}
    p \in \rho(\alpha(z_1) \cdots \alpha(z_{n-1}))U(z_n)
  \end{equation}
  for all $n$; since $U(z_n) \subset W(z_n)$, this implies the
  previous condition.

  Since $\{U(z) : z \in Z\}$ is a covering, we can find $z_1 \in Z$ so
  that $p \in U(z_1)$, so \eqref{eq:code_point} is satisfied for
  $n = 1$. Now assume inductively that we have defined a vertex path
  $z_1, \ldots, z_n$ in $\mc{G}$ so that \eqref{eq:code_point} is
  satisfied up to $n$. Since the family of sets $\{U(z)\}$ covers $M$, there is some
  vertex $z_{n+1}$ so that
  \[
    \rho(\alpha(z_1) \cdots \alpha(z_n))^{-1}p \in U(z_{n+1}).
  \]
  From \eqref{eq:code_point}, the point $\rho(\alpha(z_1) \cdots \alpha(z_n))^{-1}p$ also lies in
  $\rho(\alpha(z_n))^{-1}U(z_n)$. Thus $\rho(\alpha(z_n))^{-1}U(z_n)\cap U(z_{n+1})\neq \varnothing$. So, there is an edge from $z_n$ to $z_{n+1}$, which gives a continuation of the vertex path which still satisfies
  \eqref{eq:code_point}.
\end{proof}





\subsection{Stability criterion}

We can now state the main result of the section:
\begin{theorem}
  \label{thm:topological_criterion}
  Let $S$ be a finite generating set for $\Gamma$, and let
  $\rho_0:\Gamma \to \Homeo(M)$ be an action. Let $\mc{S}, \mc{S}'$ be
  proto-coders adapted to $\rho_0$, and suppose that all of the
  following conditions are satisfied:
  \begin{enumerate}[label=(P\arabic*)]
  \item\label{item:stable} Both proto-coders $\mc{S}$ and $\mc{S}'$
    are $\rho_0$-stable (as in \Cref{defn:stable_system}).
  \item\label{item:contracting} The point coders
    $\mc{G} = \mc{G}(\mc{S}, \rho_0)$ and
    $\mc{G}' = \mc{G}(\mc{S}', \rho_0)$ generated by $\mc{S}, \mc{S}'$
    and $\rho_0$ are both $\rho_0$-contracting (as in
    \Cref{defn:contracting_coding}).
  \item\label{item:pair_cocompactness} For every pair of distinct
    points $p, p' \in M$, there exists a group element
    $\gamma \in \Gamma$ so that for any vertices $z, z' \in Z(\mc{S})$
    satisfying $\rho_0(\gamma)p \in U(z)$ and
    $\rho_0(\gamma)p' \in U(z')$, we have
    $W(z) \cap W(z') = \emptyset$.
  \item\label{item:meandering} There is a finite subset
    $F \subset \Gamma$ satisfying the following: suppose that
    $\coding{c}_1$ and $\coding{c}_2$ are two
    $(\mc{G}, \rho_0)$-codings of the same point $p \in M$, and that
    the initial points of $\coding{c}_1$ and $\coding{c}_2$ both lie
    in $S \cup \{\id\}$. Then, there is a $\mc{G}'$-coding of $p$
    starting at $\id$ which is $F$-close to both $\coding{c}_1$ and
    $\coding{c}_2$.
  \end{enumerate}

  Then $\rho_0$ is topologically stable.
\end{theorem}

\subsection{Proof of \Cref{thm:topological_criterion}}

The proof of \Cref{thm:topological_criterion} follows the outline of
the proof of the main theorem in \cite{MannManningWeisman22}, which
essentially shows \Cref{thm:topological_criterion} in the special case
where where the proto-coders $\mc{S}$ and $\mc{S}'$ agree, and the
hypothesis \ref{item:meandering} is replaced with the (stronger)
assumption that all of the $(\mc{G}, \rho_0)$-codings are $F$-close to
each other. The more general \ref{item:meandering} is inspired by the
``meandering-hyperbolicity'' property for actions appearing in
\cite{KapovichKimLee19}.

\subsubsection{Conditions on the perturbation}
To prove \Cref{thm:topological_criterion}, we first explicitly
describe a neighborhood $\mc U$ of $\rho_0$ in
$\Hom(\Gamma, \Homeo(M))$ which will consist of representations
semi-conjugate to $\rho_0$. We want every representation
$\rho \in \mc U$ to satisfy the conditions
\ref{item:same_combinatorics}, \ref{item:perturbed_edge_condition},
and \ref{item:uniform_nesting_perturbation} described below. The first
two conditions are straightforward to state.
\begin{enumerate}[label=(C\arabic*),series=perturbation_conditions]
\item\label{item:same_combinatorics} For every pair
  $y, z \in Z(\mc{S})$, we have
  \[
    \rho(\alpha(z))^{-1}U(z) \cap U(y) = \emptyset \iff
    \rho_0(\alpha(z))^{-1}U(z) \cap U(y) = \emptyset.
  \]
\item\label{item:perturbed_edge_condition} For every edge $z \to y$ in
  either $\mc{G}$ or $\mc{G}'$, we have
  \[
    \rho(\alpha(z))\overline{W(y)} \subset W(z).
  \]
\end{enumerate}
The third condition needs a little more setup. First, assume that the
finite set $F$ in condition \ref{item:meandering} is closed under
taking inverses by (if necessary) replacing it with a larger set; this
ensures that $F$-closeness is a symmetric condition. Now, by
\Cref{lem:uniform_nesting}, there is a constant $N \in \N$ satisfying
the following. Suppose $\coding{c}$ is a $\mc{H}$-coding and
$\coding{d}$ is an $\mc{H}'$-coding for
$\mc{H}, \mc{H}' \in \{\mc{G}, \mc{G}'\}$, and that $\coding{d}$ and
$\coding{d}$ respectively have path sequences $(g_k)_{k = 0}^\infty$,
$(h_k)_{k=0}^\infty$ and terminal vertex sequences
$(z_k)_{k=1}^\infty, (y_k)_{k=1}^\infty$. Then, if $g_n^{-1}h_m \in F$
for some pair of indices $m, n \in \N$, we have
\begin{equation}
  \label{eq:uniform_nesting_initial}
  \rho_0(g_{n+N})\overline{W(z_{n+N})} \subset \rho_0(h_m)W(y_m).
\end{equation}
We impose the following condition on every $\rho \in \mc U$:
\begin{enumerate}[label=(C\arabic*),resume=perturbation_conditions]
\item\label{item:uniform_nesting_perturbation} Fix $N$ as above, and
  suppose that $\coding{c}$ and $\coding{d}$ are codings with path
  sequences $(g_k)_{k = 0}^\infty$, $(h_k)_{k=0}^\infty$ and terminal
  vertex sequences $(z_k)_{k=1}^\infty, (y_k)_{k=1}^\infty$ as
  above. Then, if $g_n^{-1}h_m \in F$ for some pair of indices
  $m, n \in \N$, we have
  \begin{equation}
    \label{eq:uniform_nesting_perturbation}
    \rho(g_{n+N})\overline{W(z_{n+N})} \subset \rho(h_m)W(y_m).
  \end{equation}
\end{enumerate}

We observe:
\begin{proposition}
  All of the conditions
  \ref{item:same_combinatorics}-\ref{item:uniform_nesting_perturbation}
  are open in $\Hom(\Gamma, \Homeo(M))$.
\end{proposition}
\begin{proof}
  Conditions \ref{item:same_combinatorics} and
  \ref{item:perturbed_edge_condition} each correspond to finitely many
  open conditions. To see that condition
  \ref{item:uniform_nesting_perturbation} does also, multiply both
  sides of \eqref{eq:uniform_nesting_initial} by $\rho_0(g_n^{-1})$
  and both sides of \eqref{eq:uniform_nesting_perturbation} by
  $\rho(g_n)^{-1}$, and notice that there are only finitely many
  possibilities for $g_n^{-1}g_{n+N}$ and $g_n^{-1}h_m$.
\end{proof}

\subsubsection{Defining the inverse of the semi-conjugacy} Now, fix $\rho \in U$ satisfying the conditions above. We want to
define an equivariant surjective map $\phi:M \to M$ intertwining the $\rho$
and $\rho_0$ actions; we will define this map by first specifying its inverse, i.e. by specifying the preimages $\phi^{-1}(p)$ for each $p \in M$. As a first step, we let $\mc{C}(\mc{G})$ denote
the set of all $\mc{G}$-codings, and similarly let $\mc{C}(\mc{G}')$
denote the set of $\mc{G}'$-codings. There is a map
\[
  \Psi:\mc{C}(\mc{G}) \sqcup \mc{C}(\mc{G}') \to 2^M
\]
defined as follows: for a coding
$\coding{c} \in \mc{C}(\mc{G}) \sqcup \mc{C}(\mc{G}')$ with path
sequence $(g_k)_{k = 0}^\infty$ and terminal vertex sequence
$(z_k)_{k=1}^\infty$, define
\[
  \Psi(\coding{c}) = \bigcap_{n = 0}^\infty\rho(g_n)W(z_n).
\]
Condition \ref{item:perturbed_edge_condition} implies that
$\Psi(\coding{c})$ is a nonempty closed subset of $M$.

\begin{lemma}
  \label{lem:perturbed_codings_agree}
  If two codings $\coding{c}$ and $\coding{d}$ in
  $\mc{C}(\mc{G}) \sqcup \mc{C}(\mc{G}')$ are $F$-close, then
  $\Psi(\coding{c}) = \Psi(\coding{d})$.
\end{lemma}
\begin{proof}
  We will show that $\Psi(\coding{c}) \subseteq \Psi(\coding{d})$; the
  other inclusion is symmetric since we have enlarged $F$ to ensure
  that $F$-closeness is symmetric. Let $(g_k)_{k=0}^\infty$ and
  $(h_k)_{k=0}^\infty$ be the path sequences for the pair of codings,
  and let $(z_k)_{k=1}^\infty$ and $(y_k)_{k=1}^\infty$ be the
  terminal vertex sequences for $\coding{c}$ and $\coding{d}$,
  respectively. It suffices to show that for infinitely many indices
  $m \in \N$, we have
  \[
    \bigcap_{n=0}^\infty \rho(g_n)W(z_n) \subseteq \rho(h_m)W(y_m).
  \]
  By assumption, for infinitely many $m$, we can find an index $n$
  such that $g_n^{-1}h_m \in F$ (see \Cref{rem:fclose_indices}). From
  condition \ref{item:uniform_nesting_perturbation}, we have
  \[
    \rho(g_{n+N})W(z_{n+N}) \subset \rho(h_m)W(y_m),
  \]
  and so we get the desired inclusion.
\end{proof}

\begin{corollary}
  \label{cor:codings_well_defined}
  Let $\coding{c}_1$ and $\coding{c}_2$ be two
  $(\mc{G}, \rho_0)$-codings of the same point $p \in M$ with start
  points $s_1, s_2 \in S \cup \{\id\}$, respectively. Then
  $\Psi(\coding{c}_1) = \Psi(\coding{c}_2)$.
\end{corollary}
\begin{proof}
  By condition \ref{item:meandering}, there is a
  $(\mc{G}', \rho_0)$-coding $\coding{d}$ of $p$, starting at the
  identity, which is $F$-close to both $\coding{c}_1$ and
  $\coding{c}_2$. Then \Cref{lem:perturbed_codings_agree} implies that
  $\Psi(\coding{c}_1) = \Psi(\coding{d}) = \Psi(\coding{c}_2)$.
\end{proof}

The corollary tells us that we may use $\Psi$ to define a map
$\Phi:M \to 2^M$. Explicitly, if $p$ is any point in $M$, then
$\Phi(p)$ is defined to be $\Psi(\coding{c})$ for any
$(\mc{G}, \rho_0)$-coding $(\id, \coding{c})$ of $p$ starting at the
identity. By \Cref{lem:coding_exists}, at least one such
$(\mc{G}, \rho_0)$-coding exists, and by
\Cref{cor:codings_well_defined}, the choice of coding does not matter,
so the map is well-defined.

We can also verify:
\begin{lemma}
  \label{lem:fibers_equivariant}
  The map $\Phi:M \to 2^M$ is $(\rho_0,\rho)$-equivariant, in the sense
  that for any $\gamma \in \Gamma$ and any $p \in M$, we have
  \[
    \Phi(\rho_0(\gamma)p) = \rho(\gamma)\Phi(p).
  \]
\end{lemma}
\begin{proof}
  It suffices to check that the claim holds when $\gamma$ lies in the
  finite generating set $S$, since then we can apply induction. So,
  assume $\gamma = s \in S$. Note that if
  \[
    \coding{c} = (\id, (e_n)_{n=1}^\infty)
  \]
  is a $(\mc{G}, \rho_0)$-coding for some $p \in M$ starting at $\id$,
  then
  \[
    \coding{c}' = (s, (e_n)_{n=1}^\infty)
  \]
  is a $(\mc{G}, \rho_0)$-coding for $\rho_0(s)p$, starting at
  $s$. Thus by \Cref{cor:codings_well_defined} we have
  $\Psi(\coding{c}') = \Phi(\rho_0(s)p)$. On the other hand, since the
  path sequence for $\coding{c}'$ is an $s$-translate of the path
  sequence for $\coding{c}$, the definition of $\Psi$ implies that
  \[
    \Psi(\coding{c}') = \rho(s)\Psi(\coding{c}).
  \]
  Then since $\Psi(\coding{c}) = \Phi(p)$, we get
  \[
    \Phi(\rho_0(s)p) = \rho(s)\Phi(p).
  \]
\end{proof}

\subsubsection{Defining the semi-conjugacy}

The next step is to show that the map $\Phi:M \to 2^M$ actually
specifies the fibers of a well-defined map $\phi:M \to
M$. Equivalently, we need to check:
\begin{proposition}
  The sets $\{\Phi(p) : p \in M\}$ give a partition of $M$, indexed by
  $M$.
\end{proposition}
\begin{proof}
  There are two things we need to verify: first, that
  \[
    \bigcup_{p \in M}\Phi(p) = M,
  \]
  and secondly, that if $p, p'$ are distinct points in $M$, then
  $\Phi(p) \cap \Phi(p') = \emptyset$.

  For the first claim, observe that conditions
  \ref{item:same_combinatorics} and
  \ref{item:perturbed_edge_condition} guarantee that the proto-coder
  $\mc{S}$ is also adapted to the perturbed action $\rho$. Moreover,
  the conditions also imply that the $\rho$-coder
  $\mc{G}(\mc{S}, \rho)$ generated by $\mc{S}$ and $\rho$ is precisely
  $\mc{G}$. So by \Cref{lem:coding_exists}, every point $p \in M$ has
  a $(\mc{G}, \rho)$-coding starting at the identity. As $\mc{G}$ is
  $\rho_0$-contracting, the corresponding infinite path in the
  underlying graph for $\mc{G}$ specifies a $(\mc{G}, \rho_0)$-coding
  (starting at the identity) of a point $q \in M$, and from the definition it follows immediately that $p \in \Phi(q)$.

  Now we turn to the second claim. Fix $p, p' \in M$ distinct. By
  condition \ref{item:pair_cocompactness}, there exists some
  $\gamma \in \Gamma$ so that for any vertices $z, z' \in Z$
  satisfying $\rho_0(\gamma)p \in U(z), \rho_0(\gamma)p' \in U(z')$,
  we have $W(z) \cap W(z') = \emptyset$; moreover, such a pair of
  vertices $z, z'$ exists because the $U(z)$ sets cover $M$. The
  construction in the proof of \Cref{lem:coding_exists} tells us that
  there exist $(\mc{G}, \rho_0)$-codings of $\rho_0(\gamma)p$ and
  $\rho_0(\gamma)p'$, each starting at the identity, whose initial
  vertices are, respectively, $z, z'$; then the definition of $\Phi$
  implies that
  \[
    \Phi(\rho_0(\gamma)p) \subset W(z), \qquad \Phi(\rho_0(\gamma)p')
    \subset W(z'),
  \]
  so in particular $\Phi(\rho_0(\gamma)p)$ and
  $\Phi(\rho_0(\gamma)p')$ are disjoint. Then
  \Cref{lem:fibers_equivariant} (equivariance) shows that the same is
  true of $\Phi(p), \Phi(p')$.
\end{proof}

The lemma tells us that we can define a map $\phi:M \to M$ with fibers
given by $\Phi$. That is, for each $p \in M$, we define
\[
  \phi(p) = q \iff p \in \Phi(q).
\]
This map is well-defined on all of $M$ by the above; it is
transparently surjective since $\Phi$ is a well-defined map on
$M$. Moreover, $(\rho_0, \rho)$-equivariance of $\Phi$ immediately
implies that $\phi$ is $(\rho, \rho_0)$-equivariant.

So, to finish the proof, we just need to show:
\begin{lemma}
  The map $\phi:M \to M$ is continuous.
\end{lemma}
\begin{proof}
  We proceed by contradiction. Suppose that there is a sequence of
  points $(p_n)$ in $M$ such that $p_n \to p$, but $\phi(p_n)$ does
  not converge to $\phi(p)$. Then since $M$ is compact, we can extract
  a subsequence so that $\phi(p_n)$ converges to some $q \ne
  \phi(p)$. Choose $\gamma \in \Gamma$ as in condition
  \ref{item:pair_cocompactness} above, and fix vertices $z, z'$ so
  that
  \[
    \rho_0(\gamma)\phi(p) \in U(z), \quad \rho_0(\gamma)q \in U(z').
  \]
  We note that \ref{item:pair_cocompactness} implies $W(z)\cap W(z')=\emptyset$.
  
  On the other hand, there are $(\mc{G}, \rho_0)$-codings of
  $\rho_0(\gamma)\phi(p)$ and $\rho_0(\gamma)q$ (starting at the
  identity) with initial vertices $z, z'$. This implies
  \[
    \Phi(\rho_0(\gamma)\phi(p)) \subset W(z), \quad
    \Phi(\rho_0(\gamma)q) \subset W(z').
  \]
  Then, by equivariance we see that
  \[
    \Phi(\phi(\rho(\gamma)p)) \subset W(z),
  \]
  and then the definition of $\phi$ implies that
  $\rho(\gamma)p \in W(z)$. Furthermore, since $\rho(\gamma)$ is a fixed
  homeomorphism of $M$, we know that $\rho(\gamma)p_n$ converges to
  $\rho(\gamma)p$, and therefore $\rho(\gamma)p_n \in W(z)$ for all
  sufficiently large $n$.

  On the other hand, since $\phi(p_n)$ converges to $q$, by
  equivariance $\phi(\rho(\gamma)p_n) = \rho_0(\gamma)\phi(p_n)$
  converges to $\rho_0(\gamma)q$. But then, for all sufficiently large
  $n$ we must have $\phi(\rho(\gamma)p_n) \in U(z')$. This implies
  that there is a $(\mc{G}, \rho_0)$-coding of $\phi(\rho(\gamma)p_n)$
  (starting at the identity) with the initial vertex $z'$, implying
  \[
    \rho(\gamma)p_n \in \Phi(\phi(\rho(\gamma)p_n)) \subset W(z').
  \]
  This contradicts $W(z) \cap W(z') = \emptyset$.
\end{proof}

\section{Background and notation for Euclidean buildings}
\label{sec:euclidean_building_background}

In this section we fix our notation and conventions for Euclidean
buildings, and prove a handful of basic geometric results. We mostly
adopt the terminology in \cite{KleinerLeeb97}, and refer to that paper
for both precise definitions of the objects appearing here and a
thorough introduction to the CAT(0) viewpoint on Euclidean
buildings.

\subsection{CAT(0) reminders}

Let $X$ be a locally compact CAT(0) space. Then $X$ has a \emph{visual
  boundary} $\vbdry X$ and, at each point $x \in X$, a \emph{space of
  directions} $\Sigma_xX$. Precisely, $\vbdry X$ is the space of
geodesic rays in $X$ modulo the relation identifying rays whose image
have finite Hausdorff distance, and $\Sigma_xX$ is the space of
geodesic segments in $X$ with an endpoint at $x$, modulo the relation
identifying segments $[x,y], [x,z]$ with vanishing \emph{Alexandrov
  angle} $\angle_x(y,z)$ at $x$. The Alexandrov angle is itself the
limit of the CAT(0) comparison angles
$\angle_{\bar{x}}(\bar{y'}, \bar{z'})$ as $y',z' \to x$ along $[x,y]$
and $[x,z]$, and defines a CAT(1) metric on $\Sigma_xX$. The visual
boundary may similarly be equipped with the \emph{Tits} or \emph{angle
  metric}: for points $\xi_1, \xi_2 \in \vbdry X$, the angular
distance $\dt(\xi_1, \xi_2)$ is the supremum over all basepoints
$x \in X$ of the Alexandrov angle between the unique representatives
of $\xi_1, \xi_2$ based at $x$. It is also possible to equip the
visual boundary with the \emph{cone topology}. 
In general, the
topology induced by the angle metric is finer than this topology.

For each $x \in X$, there is a well-defined \emph{logarithm map}
$\log_x:X \minus \{x\} \to \Sigma_xX$, taking a point
$y \in X \minus \{x\}$ to the class of the segment $[x, y]$. Often we
will write $\vec{xy}$ for $\log_x(y)$. The logarithm map extends
continuously to a surjective map $\log_x:(\vbdry X, d_T) \to \Sigma_x X$,
which is continuous (in fact, 1-Lipschitz) if $\Sigma_x X$ is equipped
with the angle metric defined by Alexandrov angles at $x$.

We frequently need the following basic CAT(0) comparison estimate:
\begin{lemma}
  \label{lem:cat0_angle_estimate}
  Let $a, b, c$ be pairwise distinct points in a CAT(0) space $X$. For
  any $\alpha \in [0,\pi]$ satisfying $\angle_b(a,c) \ge \alpha$, we
  have
  \[
    \frac{d_X(a,c)}{d_X(a,b)} \ge \sin(\alpha).
  \]
\end{lemma}
\begin{proof}
  Consider a CAT(0) comparison triangle $\bar{a}, \bar{b}, \bar{c}$
  for $a, b, c$, with Euclidean comparison angle
  $\angle_{\bar{b}}(\bar{a}, \bar{c})$ at $\bar{b}$. There are two
  cases. Suppose first that
  $\angle_{\bar{b}}(\bar{a},\bar{c}) \ge \pi/2$. Then the Euclidean
  law of cosines implies that
  $d_X(a,c)^2 \ge d_X(a,b)^2 + d_X(b,c)^2$, and in particular
  $d_X(a,c) \ge d_X(a,b)$; thus the desired inequality holds for any
  $\alpha \in [0, \pi]$.  On the other hand if
  $\angle_{\bar{b}}(\bar{a}, \bar{c}) \le \pi/2$, then we have
  \[
    \alpha \le \angle_b(a,c) \le \angle_{\bar{b}}(\bar{a}, \bar{c}) \le \pi/2,
  \]
  hence
  \[
    \sin\alpha \le \sin \angle_b(a,c) \le \sin
    \angle_{\bar{b}}(\bar{a}, \bar{c}).
  \]
  Then the Euclidean law of sines implies
  \[
    \frac{d_X(a,c)}{d_X(a,b)} = \frac{\sin \angle_{\bar{b}}(\bar{a},
      \bar{c})}{\sin \angle_{\bar{c}}(\bar{a}, \bar{b})} \ge
    \sin\angle_{\bar{b}}(\bar{a}, \bar{c}) \ge \sin\alpha.
  \]
\end{proof}

\subsection{Euclidean and spherical buildings}

Now let $X$ be a locally compact Euclidean building, i.e. a proper
CAT(0) space locally modeled on a \emph{Euclidean Coxeter complex}
$(\tamod, \tilde{W})$.  Here $\tamod$ is a \emph{model apartment}, a
Euclidean space of dimension $k = \rank(X)$, and $\tilde{W}$ is the
\emph{affine Weyl group}, a Coxeter group generated by reflections in
$\tamod$, acting with compact quotient on $\tamod$. A \emph{Euclidean
  Weyl chamber} (or a \emph{chamber} if the context is clear) in
$\tamod$ is a fundamental domain for the $\tilde{W}$-action; similarly
a chamber in $X$ is any subset of $X$ isometrically identified with a
chamber in the model apartment.

\subsubsection{Associated spherical buildings}

Both the visual boundary $\vbdry X$ and each space of directions
$\Sigma_xX$ have the structure of a \emph{spherical building} modeled
on a \emph{spherical Coxeter complex} $(\amod, W)$; here $\amod$ is a
model sphere of dimension $\rank(X) - 1$, identified with the boundary
of the model apartment $\tamod$, and $W$ is the finite reflection
group given by the image of $\tilde{W}$ under the representation
$\tilde{W} \to \Isom(\amod)$. The stabilizer in $\tilde{W}$ of any
point in $\tamod$ can be identified with a subgroup of $W$, and there
is always at least one point in $\tamod$ whose stabilizer is precisely
$W$. \emph{We will always choose basepoints in $\tamod$ (and in $X$)
  to satisfy this condition.}

The sphere $\amod$ has a simplicial
structure. A \emph{spherical Weyl chamber} (or just a \emph{chamber}
if the context is clear) in $\amod$ is a fundamental domain for the
$W$-action, i.e. a maximal simplex in $\amod$; similarly a chamber in
$\vbdry X$ is a subset isometrically identified with a chamber in
$\amod$. Note that the logarithm map takes apartments in the building
$\vbdry X$ to apartments in $\Sigma_xX$ (not necessarily injectively),
and takes chambers in $\vbdry X$ to chambers in $\Sigma_xX$
(isometrically with respect to the induced Tits and Alexandrov
distances).

The quotient $\sigmamod=\amod/W$ is called the \emph{model Weyl
  chamber} for the Coxeter complex $(\amod, W)$. Any spherical
building $S$ modeled on $(\amod, W)$ has a \emph{type map} or
\emph{anisotropy map} $\theta:S \to \sigmamod$, defined on the point
$x \in S$ by choosing an apartment $\mathcal{A} \subset S$ containing
$x$, identifying $\mathcal{A}$ with $\amod$, and mapping $x$ to its
image in the quotient $\amod \to \amod / W$. The map is independent of
the choice of apartment $\mathcal{A}$ and the choice of identification
of $\mathcal{A}$ with $\amod$. If $x$ is a point in a Euclidean
building $X$, then the logarithm map $\log_x:\vbdry X \to \Sigma_xX$
is \emph{type-preserving}, i.e. it satisfies
$\theta(\log_x(\xi)) = \theta(\xi)$. 

\subsection{Flag spaces at infinity}

The main result in this paper concerns actions on \emph{flag spaces},
which are spaces of $k$-simplices in the boundary of a Euclidean
building. To study these spaces we adopt the notation and terminology
of Kapovich--Leeb--Porti \cite{KLP2018}. Kapovich--Leeb--Porti's setup
allows for a unified study of the flag spaces at the boundary of both
Euclidean buildings and Riemannian symmetric spaces, but in this paper
we are only concerned with the former.

\subsubsection{Cones and sectors}

Whenever $A$ is a subset of $\vbdry X$, and $o \in X$, then $V(o, A)$
denotes the \emph{cone} consisting of the union of all geodesic rays
in $X$ based at $o$ and asymptotic to a point in $A$. If $\taumod$ is
a face of the simplex $\sigmamod$, and $\tau \subset \vbdry X$ is a
simplex isometric to $\taumod$ under $\theta$, then the set
$V(o, \tau)$ is called a (Euclidean) \emph{Weyl sector of type
  $\taumod$}; when $\sigma \subset \vbdry$ is a \emph{maximal}
simplex, then $V(o, \sigma)$ is simply called a (Euclidean) \emph{Weyl
  sector} (instead of Weyl sector of type $\sigmamod$). A Weyl sector
is a convex subset of an apartment in $X$.

\begin{definition}
  Let $\taumod\subset \sigmamod$ be a closed subsimplex. The
  \emph{$\taumod$-flag space} associated to $X$ is the set
  \[
    \flags=\{\tau\subset \vbdry X|\text{ $\tau$ is isometric to $\taumod$
      under the type map $\theta$}\}.
  \]
  After fixing a basepoint $o \in X$, the set $\flags$ can be
  equivariantly identified with the set of Weyl sectors $V(o, \tau)$
  of type $\taumod$, and then endowed with the cone topology. Later we
  will also define a family of metrizations on $M$, agreeing with this
  topology (which is independent of the choice of basepoint). Note
  that $M$ can also be identified with the set of equivalence classes
  of Weyl sectors of type $\taumod$, where two sectors of type
  $\taumod$ are equivalent if they have finite Hausdorff distance from
  each other.
\end{definition}


\subsection{Stars and open stars}

We now discuss the simplicial structure on the buildings $\vbdry X$
and $\Sigma_xX$ in more detail, continuing to adopt the conventions of
\cite{KLP2018}. If $\taumod$ is a face of $\sigmamod$, the \emph{open
  star} of $\taumod$, denoted $\ost(\taumod)$, is the subset of
$\sigmamod$ consisting of the union of open faces in $\sigmamod$ whose
closure contains $\taumod$. In particular, $\ost(\sigmamod)$ is equal
to $\mathrm{int}(\sigmamod)$. On the other hand, if $\taumod$ is a
vertex of $\sigmamod$, then $\ost(\taumod)$ is
$\sigmamod \minus \taumod^{\mathrm{opp}}$, where
$\taumod^{\mathrm{opp}}$ is the closed face opposite to $\taumod$. See
\Cref{fig:compact_ost}.

For a simplex $\tau$ in either $\vbdry X$ or $\Sigma_xX$ of type
$\taumod$, the \emph{star} of $\tau$, denoted $\st(\tau)$, is the
union of all simplices in $\vbdry X$ (or $\Sigma_xX$) containing
$\tau$. For any $x \in X$, the logarithm map
$\log_x:\vbdry X \to \Sigma_xX$ respects the simplicial structure,
meaning it takes stars to stars. In fact we have the following
stronger statement:
\begin{lemma}[{\cite[Lemma 2.9]{KLP2018}}]
  \label{lem:stars_to_stars}
  For every $p \in X$ and every simplex $\tau \subset \vbdry X$, we
  have
  \[
    \log_p(\st(\tau)) = \st(\log_p(\tau)).
  \]
\end{lemma}
The nontrivial part of the lemma is that the logarithm map is
surjective onto $\st(\log_p(\tau))$. That is, if $\vec{v}$ is a
direction at $p$ lying in a face adjacent to $\log_p(\tau)$, then
there is a geodesic ray $[p, \xi)$ in the direction of $\vec{v}$, with
$\xi \in \vbdry X$ lying in a face adjacent to $\tau$.

\subsection{The opposition involution}

Fix a chamber $\sigma$ in the model sphere $\amod$. The Weyl group $W$
has a unique \emph{longest word} $w_0$, taking $\sigma$ to the unique
antipodal chamber $-\sigma \subset \amod$. Identifying $\sigma$ with
the model chamber $\sigmamod$, we obtain an an involution
\[
  \iota:\sigmamod \to \sigmamod
\]
defined by $\iota(x) = -w_0x$ for all $x \in \sigmamod$. Equivalently,
since the action of the Weyl group on any apartment $\mathcal{A}$
leaves $\theta$ invariant, one may define $\iota(x)$ to be the type
$\theta(x^-)$ for any point $x^-$ such that $x, x^-$ are antipodal
points in a common apartment $\mathcal{A}$. Note that the first
description of $\iota$ shows that $\iota$ is an isometry with respect
to the Tits distance on $\sigmamod$, since both the Weyl group $W$ and
the map $x \mapsto -x$ act by isometries on the model apartment
$\amod$.

We observe:
\begin{proposition}
  \label{prop:opposition_involution}
  If $x, y \in X$ are distinct, then
  $\theta(\vec{xy}) = \iota \theta(\vec{yx})$.
\end{proposition}
\begin{proof}
  Consider an apartment $\mathcal{A} \subset X$ containing $x$ and
  $y$. The segment $[x, y]$ extends uniquely to a geodesic line
  $\ell \subset \mathcal{A}$, which can be oriented from $x$ towards
  $y$; since $\log_x$ and $\log_y$ are both type-preserving, the
  forward ideal endpoint $\ell^+$ of $\ell$ has type
  $\theta(\vec{xy})$, and the backward endpoint $\ell^-$ has type
  $\theta(\vec{yx})$. Since $\ell^+$ and $\ell^-$ are antipodal points
  in the spherical apartment $\vbdry \mathcal{A} \subset \vbdry X$, we
  have $\iota(\theta(\ell^+)) = \theta(\ell^-)$, hence $\iota
  \theta(\vec{xy}) = \theta(\vec{yx})$.
\end{proof}

An easy consequence of the above is the following:
\begin{lemma}
  \label{lem:opposite_weyl_sectors}
  Let $A$ be an apartment in $X$, and let $\tau_\pm$ be a pair of
  antipodal faces in $\vbdry A$ (i.e. $\tau_-$ consists precisely of
  the points in $\vbdry A$ antipodal to some point in $\tau_+$). Then
  for any $x, y \in A$, we have $x \in V(y, \tau_+)$ if and only if
  $y \in V(x, \tau_-)$.
\end{lemma}

\subsection{Regularity}

For $x \in X$, a direction $\vec{v} \in \Sigma_xX$ is called
\emph{$\taumod$-regular} if $\theta(\vec{v}) \in \ost(\taumod)$, and a
geodesic segment or ray based at $x$ is $\taumod$-regular if its
direction in $\Sigma_xX$ is. A $\taumod$-regular geodesic ray might
not be asymptotic to a unique maximal simplex
$\sigma \subset \vbdry X$; however, $\taumod$-regularity ensures that
the unique $\taumod$-face of $\sigma$ is well-defined, meaning that
the ray corresponds to a unique point in $\taumod$-flag space.

\subsubsection{Uniform regularity}

A sequence of $\taumod$-regular rays need not converge to a
$\taumod$-regular ray, because $\ost(\taumod)$ is not compact unless
$X$ has rank one. To guarantee good convergence properties, we ask for
our regular rays to be \emph{uniformly regular}. Precisely, fix a
compact subset $\Theta \subset \ost(\taumod)$; a direction $\vec{v}$
is \emph{$\Theta$-regular} if $\theta(\vec{v}) \in \Theta$, and
similarly for rays and segments. The \emph{$\Theta$-star} of $\tau$ is
the set
\[
  \st_{\Theta}(\tau) = \st(\tau) \cap \theta^{-1}(\Theta).
\]
A sequence $(x_n)$ is called \emph{$\Theta$-regular} if for every
$m<n$ we have $\theta(\vec{x_nx_m}) \in \Theta$. The sequence is
\emph{uniformly $\taumod$-regular} if it is $\Theta$-regular for some
compact $\Theta\subset \ost(\taumod)$.

\begin{figure}
  \centering
  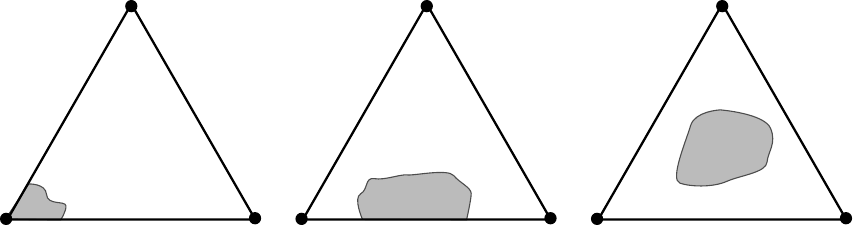
  \caption{Examples of compact subset $\Theta \subset \ost(\taumod)$
    when $\taumod$ is a vertex (left), edge (center), or interior face
    (right) of a 2-simplex $\sigmamod$. Note that the picture is not
    isometric, since $\sigmamod$ is a spherical simplex.}
  \label{fig:compact_ost}
\end{figure}

\begin{remark}
  There are slightly different conventions for $\Theta$-regularity of
  a sequence $(x_n)$ as above. One could also ask for such a sequence
  to satisfy the strictly weaker condition
  $\theta(\vec{x_0x_n}) \in \Theta$ for all $n > 0$. In
  \cite{KLP2018}, this weaker condition is called ``asymptotic uniform
  regularity'', but we warn the reader that its symmetric-space analog
  is simply called ``uniform regularity'' in
  \cite{klp2017anosov},\cite{KL2018}.
\end{remark}

The following lemma explains why uniformly regular sequences have good
convergence properties:

\begin{lemma}
  \label{lem:stars_closed}
  For any closed set $\Theta \subseteq \sigmamod$ and any
  $\taumod$-flag $\tau$, the $\Theta$-star $\st_\Theta(\tau)$ is a
  closed subset of $\vbdry X$ with respect to either the angle metric
  or the cone topology. In particular, if $X$ is locally compact, then
  $\st_\Theta(\tau)$ is compact.
\end{lemma}
\begin{proof}
We only need to consider the cone topology, since the topology induced by the angle metric is finer. Let $o \in X$ be a basepoint. Note that, if $\xi_n$ is a sequence in $\vbdry X$ converging to $\xi$ in $\vbdry X$, then the Alexandrov angles $\angle_o(\xi_n, \xi)$ converge to zero. Thus, since the map $\theta:\vbdry X \to \sigmamod$ factors through the logarithm map $\log_o:\vbdry X \to \Sigma_o X$, and $\theta$ is $1$-Lipschitz on $\Sigma_o X$, $\theta$ is continuous as a map on $\vbdry X$ equipped with the cone topology. So, we can finish the proof of the lemma by observing that $\st(\tau)$ is a closed subset of $\vbdry X$ (with respect to the cone topology); this holds because any limit of spherical Weyl chambers containing $\tau$ is also a spherical Weyl chamber containing $\tau$.
\end{proof}

Regularity also allows us to separate both the stars over
$\taumod$-flags and the cones asymptotic to those stars. Precisely, we
have the following.

\begin{lemma}[Separated stars; see {\cite[Lemma 2.4]{KLP2018}}]
  \label{lem:separated_stars}
  Let $\Theta \in \ost(\taumod)$ be compact. If $z, z'$ are distinct
  $\taumod$-flags, then for any $\xi \in \st_{\Theta}(z)$ and
  $\xi' \in \st(z')$, we have
  \[
    \angle(\xi, \xi') \ge d_{\mathrm{Tits}}(\Theta, \sigmamod \minus
    \ost(\taumod)).
  \]
\end{lemma}

\begin{corollary}
  \label{cor:theta_cones_separated}
  Let $\Theta \subset \ost(\taumod)$ be compact, let $z, z'$ be
  $\taumod$-flags, and let $x, y \in X$. If, for some $D > 0$, the
  $D$-neighborhoods of $V(x, \st_\Theta(z))$ and $V(y, \st(z'))$ have
  infinite-diameter intersection, then $z = z'$.
\end{corollary}
\begin{proof}
  Suppose that the $D$-neighborhoods of the cones
  $V(x, \st_\Theta(z))$ and $V(y, \st_\Theta(z'))$ have
  infinite-diameter intersection, and let $(x_n)$ be a sequence in
  this intersection whose distance to both $x$ and $y$ tends to
  infinity. Up to subsequence, $x_n$ converges to a point
  $\xi \in \vbdry X$, which must lie in both $\st_\Theta(z)$ and
  $\st(z')$. Thus $z = z'$ by \Cref{lem:separated_stars}.
\end{proof}

This corollary allows us to define the notion of the
$\taumod$-\emph{limit} of a sequence lying in a neighborhood of a
$\Theta$-cone.
\begin{definition}
  \label{defn:taumod_limit}
  Let $\Theta$ be a compact subset of $\ost(\taumod)$, and let $(x_n)$
  be an unbounded sequence in $X$, lying in a bounded neighborhood of
  some cone $V(o, \st_\Theta(\tau))$. Then $\tau$ is the
  \emph{$\taumod$-limit} of $x_n$.
\end{definition}

\begin{remark}
  It is possible to define $\taumod$-limits for (subsequences of)
  arbitrary $\Theta$-regular sequences, which makes the space of
  $\taumod$-flags into an actual bordification of the space $X$.
\end{remark}

\subsubsection{Stability for uniformly regular sequences}

Uniform regularity is also stable under perturbation in the following
sense: if $(x_n)$ is a $\Theta$-regular sequence, and
$\Theta' \subset \sigmamod$ is a compact set containing $\Theta$ in
its interior, then any sequence $(x_n')$ sufficiently close to $(x_n)$
will be $\Theta'$-regular. This is a direct consequence of the
following estimate:

\begin{proposition}
  \label{prop:small_type_perturbation}
  Let $x_1, x_2, y_1, y_2$ be points in $X$ satisfying
  $d_X(x_1, x_2) \le D$, $d_X(y_1, y_2) \le D$ and
  $d_X(x_1, y_1) \ge L, d_X(x_2, y_2) \ge L$. Then
  \[
    \angle(\theta(\vec{x_1y_1}), \theta(\vec{x_2y_2})) \le
    2\sin^{-1}(D/L).
  \]
\end{proposition}
\begin{proof}
  We first prove that
  $\angle(\theta(\vec{x_1y_1}), \theta(\vec{x_1y_2})) \le
  \sin^{-1}(D/L)$. Since the map $\theta:\Sigma_xX \to \sigmamod$ is
  1-Lipschitz, it suffices to bound the Alexandrov angle
  $\angle(\vec{x_1y_1}, \vec{x_1y_2})$. This is bounded by the CAT(0)
  comparison angle $\angle_{\bar{x_1}}(\bar{y_1}, \bar{y_2})$, which
  satisfies
  \[
    \angle_{\bar{x_1}}(\bar{y_1}, \bar{y_2}) \le
    \sin^{-1}\left(\frac{d_X(y_1, y_2)}{d_X(x_1, y_1)}\right) \le
    \sin^{-1}(D/L).
  \]
  as claimed. An identical argument proves that
  \[
    \angle(\theta(\vec{y_2x_1}), \theta(\vec{y_2x_2})) \le
    \sin^{-1}(D/L).
  \]
  Now, using \Cref{prop:opposition_involution} and the triangle
  inequality for angles on $\sigmamod$, we have
  \begin{align*}
    \angle(\theta(\vec{x_1y_1}), \theta(\vec{x_2y_2}))
    &\le \angle(\theta(\vec{x_1y_1}), \theta(\vec{x_1y_2})) +
      \angle(\theta(\vec{x_1y_2}), \theta(\vec{x_2y_2}))\\
    &\le \angle(\theta(\vec{x_1y_1}), \theta(\vec{x_1y_2})) +
      \angle(\iota\theta(\vec{y_2x_1}), \iota\theta(\vec{y_2x_2}))
  \end{align*}
  Since $\iota$ acts by an isometry on $\sigmamod$, this is bounded
  beneath
  \[
    \angle(\theta(\vec{x_1y_1}), \theta(\vec{x_1y_2})) +
    \angle(\theta(\vec{y_2x_1}), \theta(\vec{y_2x_2})) \le
    2\sin^{-1}(D/L),
  \]
  as desired.
\end{proof}

\subsection{Convexity and nesting}

Frequently, when considering $\Theta$-regular directions or sequences,
we will ask for the compact set $\Theta$ to be
\emph{$\taumod$-convex}. A subset $A\subset \sigmamod$ is
\emph{$\taumod$–convex} if its symmetrization
$W_{\taumod}(A) \subset \st(\sigmamod)$ is a convex subset of
$a_{mod}$; here, $W_{\taumod}$ denotes the subgroup of $W$ fixing a
face in $\amod$ of type $\taumod$ pointwise. See
\Cref{fig:taumod_convex}.

\begin{figure}[ht]
  \centering
  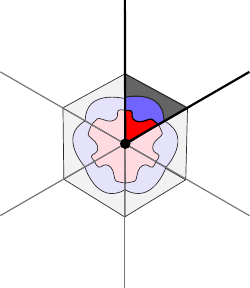
  \caption{Three compact subsets of $\ost(\taumod)$, when $\taumod$ is
    a vertex of a 2-simplex $\sigmamod$; the figure depicts a
    neighborhood of a vertex $\tau$ in the model apartment
    $\amod$. The innermost set (red) is not convex or
    $\taumod$-convex. The middle set (blue) is convex, but not
    $\taumod$-convex. The outermost set (gray) is both convex and
    $\taumod$-convex.}
  \label{fig:taumod_convex}
\end{figure}

A $\taumod$-convex subset gives rise to convex $\Theta$-stars and
$\Theta$-cones:
\begin{lemma}[{\cite[Proposition 3.16]{KLP2018}}]
  \label{lem:cone_stars_convex}
  Let $\tau$ be a $\taumod$-flag, and let
  $\Theta \subset \st(\taumod)$ be $\taumod$-convex. Then for any
  $o \in X$, the cone $V(o, \st_{\Theta}(\tau))$ is a convex subset of
  $X$.
\end{lemma}
Any compact subset $\Theta \subset \ost(\taumod)$ can be replaced with
the image of its convex hull in $\st(\taumod)$ under the type map
$\theta$, giving rise to a compact $\taumod$-convex subset of
$\sigmamod$ containing $\Theta$. So in applications, we often assume
that $\Theta$ is $\taumod$-convex.

An important consequence of \Cref{lem:cone_stars_convex} is the
following result:
\begin{lemma}[Nesting lemma for $\Theta$-cones; see {\cite[Corollary
    3.17]{KLP2018}}]
  \label{lem:theta_cone_nesting}
  Let $\Theta \subset \st(\taumod)$ be $\taumod$-convex, and let
  $\tau \in M$ satisfy $\theta(\tau) = \taumod$. If $o \in X$ and
  $p \in V(o, \st_{\Theta}(\tau))$, then
  $V(p, \st_{\Theta}(\tau)) \subseteq V(o, st_{\Theta}(\tau))$.
\end{lemma}
The nesting property of $\Theta$-cones enables local-to-global
arguments in the nonpositively curved space $X$, since the nested
cones prevent ``backtracking'' for $\Theta$-regular paths.

\subsection{The higher-rank Morse lemma}

Quasi-geodesics in rank-one symmetric spaces and Euclidean buildings
(and more generally, in Gromov-hyperbolic metric spaces) satisfy the
\emph{Morse lemma}: any $(K,A)$-quasi-geodesic stays close to a
geodesic ray, with a bound depending only on the constants $K,
A$. Thus a quasi-geodesic sequence in such a space always has a
well-defined ``endpoint'' in the boundary at infinity.

When $X$ is a Euclidean building or Riemannian symmetric space with
rank at least two, the Morse lemma fails (in fact it already fails in
any isometrically embedded 2-flat in $X$). However, while
\emph{arbitrary} quasi-geodesics may be poorly behaved, work of
Kapovich--Leeb--Porti \cite{KLP2018} shows that \emph{uniformly
  $\taumod$-regular} quasi-geodesics still have well-defined
``endpoints'' in the space of $\taumod$-flags in $\vbdry
X$. Precisely, Kapovich--Leeb--Porti proved the following higher-rank
version of the Morse lemma:

\begin{theorem}[Higher-rank Morse Lemma for Euclidean buildings; see
  {\cite[Theorem 1.3]{KLP2018}}]
  \label{thm:morselemma}
  Given a compact subset $\Theta \subset \ost(\taumod)$ and constants
  $K \ge 1, A \ge 0$, there exists a constant $D \ge 0$ satisfying the
  following. If $(x_n)$ is any $\Theta$-regular $(K,A)$-quasi-geodesic
  sequence in $X$, then there exists a unique flag $\tau \in \flags$
  so that each $x_n$ lies within a $D$-neighborhood of
  $V(x_0, \st(\tau))$.
\end{theorem}

One can improve the statement slightly. The corollary below shows that
the sequence $(x_n)$ actually must lie within a uniform neighborhood
of a $\Theta'$-cone, for some compact $\Theta' \subset \ost(\taumod)$.

\begin{corollary}
  \label{cor:improved_morse}
  Let $\Theta, K, A$ be as in \Cref{thm:morselemma}, and let $\Theta'$
  be any compact subset of $\ost(\taumod)$ whose interior contains
  $\Theta$. Then there exists some constant $D > 0$ so that for any
  $n < m$, the point $x_m$ lies in the $D$-neighborhood of
  $V(x_n, \st_{\Theta'}(\tau))$.
\end{corollary}

Combining this corollary with \Cref{cor:theta_cones_separated} shows
that uniformly regular quasi-geodesic sequences have well-defined
limits in the space of $\taumod$-flags (see \Cref{defn:taumod_limit}).

\begin{proof}[Proof of \Cref{cor:improved_morse}]
  Fix some $n \in \N$. Since the tail $(x_j)_{j=n}^\infty$ is also a
  $\Theta$-regular $(K,A)$-quasi-geodesic sequence,
  \Cref{thm:morselemma} tells us that there is some constant $D'$ and
  some $\tau_n$ such that for all $m > n$, the point $x_m$ lies in the
  $D'$-neighborhood of $V(x_n, \st(\tau_n))$. The uniqueness part of
  \Cref{thm:morselemma} implies that actually $\tau_n = \tau_0 = \tau$
  for all $n$.

  Choose $\delta > 0$ so that the closed $\delta$-neighborhood of
  $\Theta$ lies in $\Theta'$, and let $L = D'/\sin(\delta/2)$. Fix
  $m > n$; we wish to show that $x_m$ is uniformly close to the cone
  $V(x_n, \st(\tau))$. There are two cases: either
  $m - n > K(L + D' + A)$, or $m - n \le K(L + D' + A)$.

  For the first case, let $x_m'$ be the nearest-point projection of
  $x_m$ to $V(x_n, \st(\tau))$. In this case, since
  $(x_j)_{j=0}^\infty$ is a $(K,A)$-quasi-geodesic sequence, the
  distance between $x_m$ and $x_n$ is greater than $L + D'$. So,
  \Cref{prop:small_type_perturbation} implies that
  $\angle(\theta(\vec{x_nx_m}), \theta(\vec{x_nx_m'})) < \delta$;
  since we have assumed that the sequence $(x_n)$ is $\Theta$-regular,
  it follows that $\theta(\vec{x_nx_m'}) \in \Theta'$. Thus
  $x_m' \in V(x_n, \st_{\Theta'}(\tau))$, and we can conclude since
  $x_m$ is $D'$-close to $x_m'$.

  For the second case, let $M$ be the smallest integer satisfying
  $M - n > K(L + D' + A)$. Since $(x_j)_{j=0}^\infty$ is a
  $(K,A)$-quasi-geodesic and $n < m < M$, the distance between $x_m$
  and $x_M$ is at most $K^2(L + D' + A) + A$, and the first case above
  shows that $x_M$ is $D'$-close to a point in
  $V(x_n, \st_{\Theta'}(\tau))$. So in this case $x_m$ is within
  distance $K^2(L + D' + A) + A + D'$ of
  $V(x_n, \st_{\Theta'}(\tau))$.
\end{proof}

\begin{remark}
  It is possible to further improve the statement of the corollary as
  follows: for any $\Theta, \Theta', K, A$ as in the statement, there
  is a constant $D > 0$ so that any $\Theta$-regular
  $(K,A)$-quasi-geodesic sequence $(x_n)$ is within Hausdorff distance
  $D$ of another sequence $(x_n')$ satisfying
  $x_m' \in V(x_n', \st_{\Theta'}(\tau))$ for all $n < m$. In the case
  where $X$ is a symmetric space, rather than a locally compact
  Euclidean building, a statement similar to this stronger one is
  proved in \cite{KLP_local_to_global}.
\end{remark}

\section{Expansivity}
\label{sec:expansivity}

To apply the results of \Cref{sec:general_stability} towards the main
theorem, we need to construct a pair of proto-coders adapted to the
action of some cocompact group $\Gamma < \Isom(X)$ on a flag space
$\flags$ for a Euclidean building $X$. The proto-coders express the
fact that the $\Gamma$-action on $M$ is ``topologically expansive;''
to construct the coders, we will use the fact that there is a
metrization of $M$ so that the action is ``expansive'' in a more
typical sense. In particular, we will show that for every point
$\tau \in M$, there is a neighborhood $W$ of $\tau$, a group element
$\alpha \in \Gamma$, and a constant $E > 1$ so that for all
$x, y \in W$, we have
\[
  d(\alpha x, \alpha y) \ge E \cdot d(x, y)
\]
for a fixed metric $d$ on $M$. The purpose of this section is to prove
a carefully quantified version of such an expansivity result, stated
(in an inverted ``contracting'' form) as \Cref{prop:expansivity}
below.

\subsection{Metrics on flag spaces}

Before stating the result, we identify the metric on $M$ witnessing
the expansivity of the action. We actually have a family of metrics
$d_o$ on $M$, varying according to a choice of basepoint $o$ in
$X$. As mentioned previously, we will always assume that $o$ is chosen
so that its stabilizer in any apartment is identified with the full
spherical Weyl group $W$.

For each $x \in \flags$, let $b_x\in \partial_\infty X$ denote the
barycenter of the simplex associated to $x$ in $\vbdry X$. For
$\zeta\in \vbdry X$ and $o\in X$, let $[o,\zeta)$ denote the unique
geodesic ray in $X$ based at $o$ and asymptotic to $\zeta$. Then, for
every $x,y\in \flags$, define
\[
  d_o(x, y) =
  \begin{cases}
    0 &\text{ if }x = y,\\
    {2^{-|[o, b_x) \cap [o, b_y)|}}&\text{otherwise.}
\end{cases}
\]
Here $|[o, \zeta_1) \cap [o, \zeta_2)|$ denotes the length of the
(possibly trivial) geodesic segment $[o, \zeta_1) \cap [o,
\zeta_2)$. The metric $d_o$ is an \emph{ultrametric,} i.e. for every
$x,y,z\in M$, we have
\[
  d_o(x,z)\le \max\{d_o(x,y),d_o(y,z)\}.
\]

\subsection{The contraction estimate}

For the estimate below, fix a basepoint $o \in X$ and a model face
$\taumod \subseteq \sigmamod$. Let $M$ be the space of
$\taumod$-flags.

\begin{proposition}
  \label{prop:expansivity}
  Given $D > 0, E > 1$ and a compact subset
  $\Theta \subset \ost(\taumod)$, there exist $L > 0$ (depending only
  on $D, E, \Theta$) and $\eps > 0$ (depending only on $D$) so that
  the following holds. If
  $z \in M$ and $\gamma \in \Isom(X)$ satisfy:
    \begin{enumerate}[label=(\roman*)]\label{contracting-conditions}
    \item\label{item:long_translation} $d_X(o, \gamma o) > L$, and
    \item\label{item:close_to_theta_star} $\gamma o$ belongs to
      the $D$-neighborhood of the $\Theta$-cone
      $V(o,\st_{\Theta}(z))$,
  \end{enumerate}
  then $\gamma$ is $E^{-1}$-contracting on the ball
  $B(\gamma^{-1}\tau, \eps)$ with respect to the metric $d_o$. That
  is, $d_o(\gamma x, \gamma y)< \frac{1}{E}d_o(x,y)$ for every
  $x,y\in B(\gamma^{-1}z,\eps)$.
\end{proposition}

\begin{remark}
  In the analogous situation where $M$ is a flag boundary for some
  Riemannian symmetric space, a related expansivity estimate also
  appears in work of Kapovich--Leeb--Porti (see \cite[Theorem
  2.41]{klp2017anosov}). Their result, unlike ours, is infinitesimal
  in nature, and its proof relies on the fact that the flag spaces at
  the boundary of a symmetric space have the structure of a smooth
  manifold, with the isometry group acting by diffeomorphisms. In
  contrast, the statement above is non-infinitesimal; this is partly
  because the flag space $M$ is no longer a manifold, but it is also
  because our application requires us to estimate the size of an
  ``expanded'' ball. This may not be easy to do using a purely
  infinitesimal statement.
\end{remark}

Intuitively, the estimate in \Cref{prop:expansivity} holds because
$d_o$ behaves like a visual metric: as we move a basepoint in $X$
``closer'' to the ideal endpoints of a pair of geodesic rays, the
simplices containing those endpoints look farther apart in the metric
determined by the new basepoint. We first show that it is possible to
make this intuition precise in the case where one basepoint lies in
the intersection of a pair of $\sigmamod$-Weyl sectors based at the
other. This gives a sharper (but more restrictive) version of
\Cref{prop:expansivity}.

\begin{lemma}[Change-of-basepoint map in Weyl sectors]
  \label{lem:weylchamber_basepoint}
  Let $o \in X$, let $x,y \in \flags$, and let $\Theta$ be a compact
  subset of $\ost(\taumod)$. Suppose that the geodesics $[o, b_x)$ and
  $[o, b_y)$ are contained in closed $\sigmamod$-Weyl sectors
  $W_1, W_2$ with tip at $o$, and $w \in W_1 \cap W_2$ satisfies
  $\theta(\vec{ow}) \in \Theta$. There exists a constant $K > 0$,
  depending only on $\Theta$, so that
  \[
    d_o(x,y) \le 2^{-K d_X(o, w)} d_w(x,y).
  \]
\end{lemma}
\begin{proof}
  The lemma is immediate when $x = y$, so assume $x\neq y$. Let $p$ be
  the first point where $[o, b_x)$ and $[o, b_y)$ diverge, and let
  $p'$ be the first point where $[w, b_x)$ and $[w, b_y)$
  diverge. Both points $p, p'$ must lie in the boundary of the
  intersection $W_1 \cap W_2$. This intersection is a convex (not
  necessarily bounded) polyhedron $P$ in an apartment $\mc{A}$, with a
  vertex at $o$. Since $\mc{A}$ is an apartment, we can identify it
  with $\R^k$, and assume that $o$ lies at the origin.

  Let $W < \tilde{W}$ be the spherical Weyl group (which fixes $o$),
  and choose a positive root system $\{\phi_j\} \subset \mc{A}^*$ for
  $W$. Since the polyhedron $P$ is contained in a Euclidean Weyl
  sector in $\mc{A}$, up to the action of $W$ it is contained in the
  intersection of half-spaces
  \[
    V = \{x \in \mc{A} : \phi_j(x) \ge 0 \textrm{ for all } \phi_j\}.
  \]
  In fact, $P$ is exactly equal to an intersection of $V$ with
  finitely many half-spaces of the form
  $\{x\in \mc{A} : \phi_j(x) \le a_j\}$, with $a_j \ge 0$. The point
  $p$ must lie on the boundary of one of these half-spaces, meaning
  there is some positive root $\phi_j$ such that
  $\phi_j(p) = a_j < \infty$.

  Note that if $p=o$, then $w=o$ and the lemma is immediate, so assume
  that $p\neq o$. In this case we may choose the positive root
  $\phi_j$ so that $\phi_j(p) = a_j > 0$; see
  \Cref{fig:weylchamber_intersect} for a depiction of the
  situation. Now, since $p' \in P$, we must have $\phi_j(p') \le a_j$,
  hence
  \[
    \phi_j(p) \ge \phi_j(p') = \phi_j(p' - w) + \phi_j(w).
  \]
  
  \begin{figure}
    \centering
    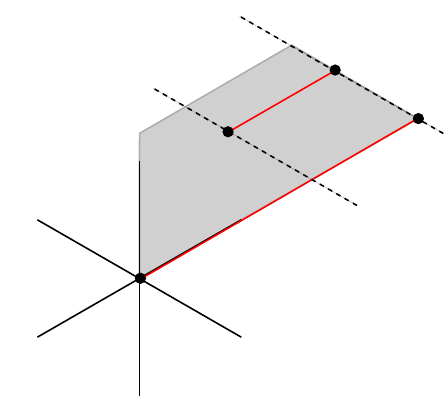
    \caption{Illustration for the proof of
      \Cref{lem:weylchamber_basepoint}. In this case $\taumod$ is a
      vertex of the one-dimensional simplex $\sigmamod$.}
    \label{fig:weylchamber_intersect}
  \end{figure}
  
  The vectors $p$ and $p' - w$ are parallel, lying in the direction of
  the barycenter of a common simplex $\tau$ of type $\taumod$ in
  $\vbdry \mc{A}$. This direction is transverse to $\ker(\phi_j)$, so
  there is a constant $C > 0$ (depending only on $\taumod$) so that
  \[
    C \cdot \phi_j(p) = d(o, p), \quad C \cdot \phi_j(p' - w) = d(w,
    p').
  \]
  Thus we have
  \[
    d(o, p) \ge d(w, p') + C \cdot \phi_j(w).
  \]
  Now, since $\phi_j$ is positive on the direction of the barycenter
  of $\tau \subset \vbdry \mc{A}$, it is also positive on vectors in
  the direction of the open star of $\tau$ in $\vbdry \mc{A}$, and
  uniformly positive on unit vectors in the direction of any compact
  subset of $\ost(\tau)$. We have assumed that $\theta(\vec{ow})$ lies
  in the compact set $\Theta$ and that $w \in V$, so this means that
  there is some uniform constant $K > 0$ so that
  $C \cdot \phi_j(w) \ge K d(o,w)$; putting this together with the
  above, we obtain
  \[
    d(o,p) \ge d(w, p') + K d(o, w),
  \]
  which gives the desired inequality.
\end{proof}

It is possible to use the estimate above (together with the nestedness
property for $\Theta$-cones, see \Cref{lem:theta_cone_nesting}) to
prove the contraction property in \Cref{prop:expansivity} in the
special case where the contracting element $g$ takes a basepoint $o$
directly into the cone $V(o, \st(\tau))$. To prove the general
contraction estimate, however, we need a version of the nesting lemma
which allows us to compare a pair of cones $V(o, \st(\tau))$ and
$V(x, \st(\tau))$ when $x$ does \emph{not} lie in $V(o,
\st(\tau))$. First, we need a way to estimate the amount of time it
takes a geodesic ray asymptotic to a Weyl cone to actually enter the
cone. This is closely related to (but not exactly the same as) an
estimate in work of Kapovich--Leeb--Porti; compare the next two
results to \cite[Prop. 3.44 and Cor. 3.45]{KLP2018}.

\begin{lemma}
  \label{lem:geodesics_enter_cones}
  Let $\alpha$ be the angle between the barycenter of $\taumod$ and
  $\sigmamod \minus \ost(\taumod)$, and let $\tau$ be a $\taumod$-flag. Fix
  $x \in X$, and let $q \in [x, b_\tau)$ be the point satisfying
  $[q, b_\tau) = [x, b_\tau) \cap V(o, \st(\tau))$. Then we have
  \[
    d(x,q) \le \frac{1}{\sin\alpha} \cdot d(x, V(o, \st(\tau)).
  \]
\end{lemma}

\begin{proof}
  Let $p$ be the nearest-point projection from $x$ to
  $V(o, \st(\tau))$. By \Cref{lem:cat0_angle_estimate}, it suffices to
  prove the inequality $\angle_q(p, x) \ge \alpha$. That is, if we
  view $\vec{qp}$ and $\vec{qx}$ as points in the spherical building
  $\Sigma_qX$, we wish to show
  $\angle(\vec{qp}, \vec{qx}) \ge \alpha$. Thus, since
  $\theta(\vec{qp}) = \iota(\theta(\vec{pq}))$ is the barycenter of
  $\iota(\taumod)$, and $\iota$ is an isometry, it suffices to show
  that $\vec{qp}$ is not contained in $\ost(\vec{qx})$. Suppose for a
  contradiction that this is not the case, i.e. that
  $\vec{qp} \in \ost(\vec{qx})$. This implies that any closed
  spherical Weyl chamber $\sigma \subset \Sigma_qX$ containing
  $\vec{qp}$ must also contain $\vec{qx}$.

  Recall that $V(o, \st(\tau))$ is a union of $\sigmamod$-Weyl sectors
  $V(o, \sigma)$ for all simplices $\sigma \subset \vbdry X$
  containing $\tau$. Each such sector $V(o, \sigma)$ is itself a union
  of Euclidean Weyl chambers, i.e. fundamental domains for the action
  of the affine Weyl group $\tilde{W}$ on an apartment containing
  $V(o, \sigma)$. Now, since $V(o, \st(\tau))$ is convex, we have
  $[q, p] \subset V(o, \st(\tau))$, and therefore an initial
  nontrivial sub-segment of $[q, p]$ is contained in the closure of
  such a chamber $C \subset V(o, \sigma)$. See
  \Cref{fig:geodesic_cone_figure}.

  \begin{figure}[ht]
    \centering
\begingroup%
  \makeatletter%
  \providecommand\color[2][]{%
    \errmessage{(Inkscape) Color is used for the text in Inkscape, but the package 'color.sty' is not loaded}%
    \renewcommand\color[2][]{}%
  }%
  \providecommand\transparent[1]{%
    \errmessage{(Inkscape) Transparency is used (non-zero) for the text in Inkscape, but the package 'transparent.sty' is not loaded}%
    \renewcommand\transparent[1]{}%
  }%
  \providecommand\rotatebox[2]{#2}%
  \newcommand*\fsize{\dimexpr\f@size pt\relax}%
  \newcommand*\lineheight[1]{\fontsize{\fsize}{#1\fsize}\selectfont}%
  \ifx\svgwidth\undefined%
    \setlength{\unitlength}{255.72257455bp}%
    \ifx\svgscale\undefined%
      \relax%
    \else%
      \setlength{\unitlength}{\unitlength * \real{\svgscale}}%
    \fi%
  \else%
    \setlength{\unitlength}{\svgwidth}%
  \fi%
  \global\let\svgwidth\undefined%
  \global\let\svgscale\undefined%
  \makeatother%
  \begin{picture}(1,0.62243305)%
    \lineheight{1}%
    \setlength\tabcolsep{0pt}%
    \put(0,0){\includegraphics[width=\unitlength,page=1]{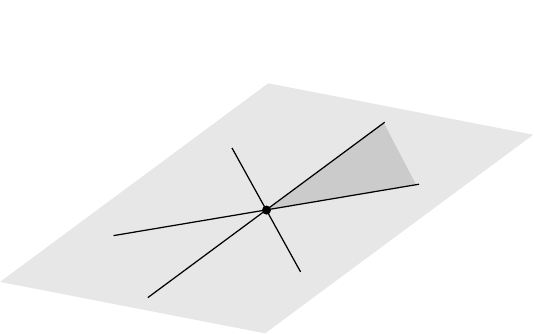}}%
    \put(0.63745245,0.31252967){\color[rgb]{0,0,0}\makebox(0,0)[lt]{\lineheight{1.25}\smash{\begin{tabular}[t]{l}$q$\end{tabular}}}}%
    \put(0.52725322,0.2018911){\color[rgb]{0,0,0}\makebox(0,0)[lt]{\lineheight{1.25}\smash{\begin{tabular}[t]{l}$o$\end{tabular}}}}%
    \put(0.54065099,0.59224945){\color[rgb]{0,0,0}\makebox(0,0)[lt]{\lineheight{1.25}\smash{\begin{tabular}[t]{l}$x$\end{tabular}}}}%
    \put(0,0){\includegraphics[width=\unitlength,page=2]{anglebound_1.pdf}}%
    \put(0.59749931,0.26065004){\color[rgb]{0,0,0}\makebox(0,0)[lt]{\lineheight{1.25}\smash{\begin{tabular}[t]{l}$p$\end{tabular}}}}%
    \put(0,0){\includegraphics[width=\unitlength,page=3]{anglebound_1.pdf}}%
    \put(0.65186862,0.40047657){\color[rgb]{0,0,0}\makebox(0,0)[lt]{\lineheight{1.25}\smash{\begin{tabular}[t]{l}$b_\tau$\end{tabular}}}}%
    \put(0.73364271,0.3189518){\color[rgb]{0,0,0}\makebox(0,0)[lt]{\lineheight{1.25}\smash{\begin{tabular}[t]{l}$C$\end{tabular}}}}%
  \end{picture}%
\endgroup%

    \caption{Illustration for the proof of
      \Cref{lem:geodesics_enter_cones}. As in
      \Cref{fig:weylchamber_intersect}, in this case $\tau$ is a
      vertex in a maximal one-dimensional simplex $\sigma$.}
    \label{fig:geodesic_cone_figure}
  \end{figure}
  
  Since $C$ is a Euclidean Weyl chamber whose closure contains $q$,
  the image of the logarithm map $\log_q:C \minus \{q\} \to \Sigma_qX$
  contains a closed maximal simplex in $\Sigma_qX$ containing
  $\vec{qp}$. Since we have assumed $\vec{qp} \in \ost(\vec{qx})$, we
  see that $\vec{qx}$ must also be in the image of this logarithm
  map. But this implies that a non-trivial initial subsegment of
  $[q, x]$ is contained in $C$, which is impossible since $C$ is a
  subset of $V(o, \sigma) \subset V(o, \st(\tau))$ and $q$ is the
  first point where the geodesic $[x, b_\tau)$ intersects
  $V(o, \st(\tau))$.
\end{proof}

As a consequence of the previous lemma, we get a similar estimate for
the amount of time it takes a geodesic ray to enter certain
$\Theta$-cones. This result will be useful later, in
\Cref{sec:interpolation}.
\begin{corollary}\label{cor:geodesics_enter_cones}
  Let $\alpha$ be the angle between the barycenter of $\taumod$ and
  $\sigmamod \minus \ost(\taumod)$, and let $\Theta$ be a subset of
  $\sigmamod$ containing a $\delta$-neighborhood of the barycenter of
  $\taumod$. Fix a point $x \in X$ and a $\taumod$-flag $\tau$, and
  let $q' \in [x, b_\tau)$ be the point satisfying
  $[q', b_\tau) = [x, b_\tau) \cap V(o, \st_\Theta(\tau))$. Then we
  have
  \[
    d_X(x,q') \le
    \frac{1}{\sin\alpha}\left(1+\frac{1}{\sin\delta}\right) d_X(x, o).
  \]
\end{corollary}
Note that $\delta \le \alpha$, so the statement also gives a bound in
terms of $\delta$ only.
\begin{proof}
  Let $q\in [x,b_\tau)$ be the point satisfying
  $[q, b_\tau) = [x, b_\tau) \cap V(o, \st(\tau))$, as in the
  statement of \Cref{lem:geodesics_enter_cones}. The proof of
  \Cref{lem:geodesics_enter_cones} shows that
  $\angle_q(x,o)\ge \alpha$. Thus, by \Cref{lem:cat0_angle_estimate},
  we have $d_X(x,q)\le \frac{d_X(x,o)}{\sin\alpha}$ and
  $d_X(o,q)\le \frac{d_X(x,o)}{\sin\alpha}$.
  
  To get a bound on $d_X(x,q')$, first note that $q'\in [q,b_\tau)$
  and $d_X(x,q')=d_X(x,q)+d_X(q,q')$. If $q=q'$ then we are done by
  the previous estimates, so assume that $q\neq q'$. Since
  $q\in V(o,\st(\tau))$, there exists a Weyl sector based at $o$
  containing $q$ and $[o,b_\tau)$. The ray $[q,b_\tau)$ must also
  belong to this sector, and be parallel to the ray
  $[o,b_\tau)$. Therefore, $\angle_{q'}(q,o)=\angle_o(q',b_\tau)$.

  On the other hand, since $q\notin V(o,\st_\Theta(\tau))$ and $q'$ is
  the first point on $[q,b_\tau)$ belonging to
  $V(o,\st_\Theta(\tau))$, we deduce that
  $\theta(\vec{oq'})\in \partial \Theta$. Then, since $\Theta$
  contains a $\delta$-neighborhood of the barycenter of $\taumod$, we
  must have $\angle_o(q',b_\tau)\ge \delta$, hence
  $\angle_{q'}(q,o)\ge \delta$. Consider the triangle $(o,q,q')$; by
  \Cref{lem:cat0_angle_estimate}, we get
  $d_X(q,q')\le \frac{d_X(o,q)}{\sin\delta}$. Therefore
  \begin{align*}
    d_X(x,q') &= d_X(x,q)+d_X(q,q') \\
              &\le \frac{d_X(x,o)}{\sin\alpha} + \frac{d_X(o,q)}{\sin\delta}\\
              &\le \frac{d_X(x,o)}{\sin\alpha}+\frac{d_X(x,o)}{\sin\alpha\sin\delta}.
  \end{align*}
\end{proof}

The following lemma shows that the nesting property for Weyl sectors
(see \Cref{lem:theta_cone_nesting}) can be transferred to a pair of
``nearby'' Weyl sectors.
\begin{lemma}
  \label{lem:weyl_cones}
  Fix points $x, y \in X$ and $\taumod$-simplices
  $\tau, \tau' \subset \vbdry X$, with respective barycenters
  $\xi, \xi'$. Suppose that $[y, \xi) \cap [y, \xi')$ contains a
  nontrivial segment. Then, if the type $\taumod$-Weyl sector
  $V(x, \tau)$ contains $[y, \tau)$, the type $\taumod$-Weyl sector
  $V(x, \tau')$ contains $[y, \tau')$.
\end{lemma}
\begin{proof}
  Let $\sigma_+ \subset \vbdry X$ be a maximal face containing
  $\tau$. Let $\mc{A}$ be an apartment containing $V(x, \sigma_+)$,
  and let $\sigma_-$ and $\tau_-$ be the antipodal faces to $\sigma_+$
  and $\tau$ in $\vbdry \mc{A}$ respectively. Since
  $y \in V(x, \tau)$, by \Cref{lem:opposite_weyl_sectors} we have
  $x \in V(y, \tau_-)$.

  Now let $z$ be a point in the intersection $[y, \xi) \cap [y, \xi')$
  such that the segment $[y, z]$ is nontrivial. We know that
  $\theta(\vec{yz}) = \theta(\xi)$ is the barycenter of $\taumod$, so
  let $\tau(\vec{yz}) \subset \Sigma_y\mc{A}$ denote the
  type-$\taumod$ simplex whose barycenter is the direction $\vec{yz}$,
  then $\tau(\vec{yz}) = \log_y(\tau) = \log_y(\tau')$. By
  \Cref{lem:stars_to_stars}, we have
\begin{equation}
  \label{eq:stars_agree}
  \log_y\st(\tau) = \log_y\st(\tau') = \st(\tau(\vec{yz})),
\end{equation}
where $\st(\tau(\vec{yz}))$ is the star in $\Sigma_yX$. Now let $\vec{v}_+ \in \Sigma_y\mc{A}$ be the the barycenter of
$\log_y(\sigma_+)$, and let $\vec{v}_-$ be its opposite direction in
the apartment $\Sigma_y\mc{A}$. Since
$\vec{v}_+\in \log_y\st(\tau) = \log_y\st(\tau')$, there exists a
$\sigmamod$-simplex $\sigma_+' \subset \st(\tau')$ such that
$\theta(\vec{v}_+)\in \theta(\sigma_+)$.

In $\Sigma_yX$, since the barycenters of the maximal simplices
$\log_y(\sigma_+')$ and $\log_y(\sigma_+)$ agree, we have
$\log_y(\sigma_+') = \log_y(\sigma_+)$, and therefore the Weyl sector
$V(y, \sigma_+')$ is opposite to the Weyl sector $V(y,
\sigma_-)$. Consequently these Weyl sectors (and their closures) are
contained in a common apartment $\mc{A}'$, with $\sigma_+'$ and
$\sigma_-$ opposite faces in $\vbdry \mc{A}'$. It follows that the
$\taumod$-face of $\sigma_+'$ (namely $\tau'$) is opposite to
$\tau_-$. Then by \Cref{lem:opposite_weyl_sectors}, since
$x \in V(y, \tau_-)$, we have $y \in V(x, \tau')$. Since $V(x, \tau')$
is convex and $\xi' \in \tau'$, this completes the proof.
\end{proof}

Finally we can put everything together and prove the main result of
the section.

\begin{proof}[Proof of \Cref{prop:expansivity}]
 

  Let $x,y\in B(\gamma^{-1}z,\eps_0)$, and let
  $\xi,\eta_1,\eta_2\in \vbdry X$ be the barycenters of the
  $\taumod$-simplices $z$, $x$, and $y$ respectively. By definition,
  the rays $[o,\gamma^{-1}\cdot\xi)$ and $[o,\eta_i)$ overlap in a segment
  of length at least $-\log_2\eps$, for $i=1,2$. The rays $[o,\eta_1)$
  and $[o,\eta_2)$ overlap in a segment of length equal to
  $-\log_2 d(\eta_1,\eta_2)\ge -\log_2\eps$. Since $\gamma$ is an
  isometry on $X$, we know that for $i = 1,2$ the rays
  $[\gamma\cdot o,\xi)$ and $[\gamma\cdot o,\gamma\cdot\eta_i)$
  overlap in a segment of length at least $-\log_2\eps$, and the rays
  $[\gamma\cdot o,\gamma\cdot \eta_1)$ and
  $[\gamma\cdot o,\gamma\cdot \eta_2)$ overlap in a segment of length
  equal to $-\log_2 d(\eta_1,\eta_2)\ge -\log_2\eps$.

  The ray $[\gamma\cdot o, \xi)$ is eventually parallel to $[o, \xi)$,
  so it eventually lies in some Euclidean Weyl sector $W$ containing
  the $\taumod$-Weyl sector $V(o, z)$. Note that $W$ has the form
  $V(o,u)\subset V(o,\st(z))$ for some $\sigmamod$-simplex $u$
  containing $z$. Let $w \in [\gamma\cdot o, \xi)$ be the first point
  where $[\gamma\cdot o, \xi)$ intersects $W$, so that
  $[\gamma\cdot o, \xi) \cap W = [w, \xi)$; see
  \Cref{fig:contraction_estimate}.

  \begin{figure}[ht]
    \centering
    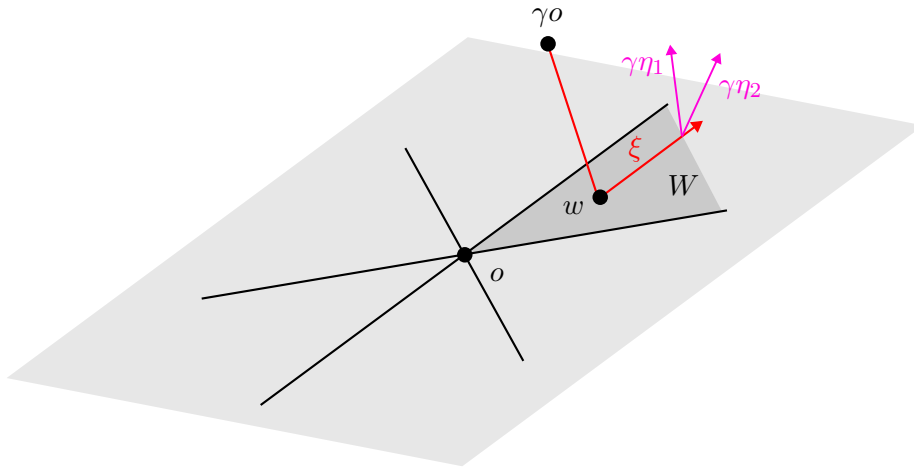
    \caption{Illustration for the proof of \Cref{prop:expansivity}.}
    \label{fig:contraction_estimate}
  \end{figure}
  
  Fix a compact subset
  $\Theta'' \subset \ost(\taumod)$ whose interior contains
  $\Theta'$. We claim that if $L$ is sufficiently large, then
  $\theta(ow) \in \Theta''$.

  To see this, fix a point $w' \in V(o, \st_{\Theta'}(z))$ with
  $d_X(\gamma \cdot o, w') \le D$. Then by
  \Cref{lem:geodesics_enter_cones} we have
  \begin{align}
      \label{eq:overlapbound}
    d_X(\gamma \cdot o, w) \le C d_X(\gamma \cdot o, \st(z)) \le C
    \cdot d_X(\gamma \cdot o, w') \le CD,
  \end{align}
  where $C$ is the constant from
  \Cref{lem:geodesics_enter_cones}. Thus $d_X(w, w') \le CD + D$. On
  the other hand, we have
  \begin{align}
    \label{eq:go_w_close}
    \angle_o(w, w') \le \angle_{\bar{o}}(\bar{w}, \bar{w}') \le
    \sin^{-1}\left(\frac{CD + D}{d_X(o, w')}\right) \le
    \sin^{-1}\left(\frac{CD + D}{L - D}\right).
  \end{align}
  So, since $\theta(ow') \in \Theta'$ by definition, and
  $\theta:\Sigma_oX \to \sigmamod$ is $1$-Lipschitz, if $L$ is large
  enough then $\theta(ow) \in \Theta''$.

  In addition, \eqref{eq:overlapbound} also implies that whenever
  $-\log_2 \eps > CD$, the geodesic segment $[\gamma\cdot o,w]$ is
  properly contained in
  $[\gamma\cdot o,\gamma\cdot \eta_1)\cap [\gamma\cdot
  o,\gamma\cdot \eta_2)$. It follows that
  \begin{align*}
  \log_2d_{\gamma\cdot o}(\gamma x, \gamma y) &=- |[\gamma\cdot o, \gamma \eta_1) \cap [\gamma\cdot o, \gamma \eta_2)| = -d_X(\gamma\cdot o,w)-|[w, \gamma \eta_1) \cap [w, \gamma \eta_2)| \\
  &\ge -CD - |[w, \gamma \eta_1) \cap [w, \gamma \eta_2)]|\\
  &= -CD +  \log_2d_w(\gamma x, \gamma y).
  \end{align*}
  Furthermore, the geodesic ray $[w, \xi)$ intersects both of the rays
  $[w, \gamma\cdot \eta_1)$ and $[w, \gamma\cdot \eta_2)$ in
  nontrivial segments. Then, as a consequence of
  \Cref{lem:weyl_cones}, the Weyl cones $W_1 = V(o, \gamma x)$
  and $W_1 = V(o, \gamma y)$ contain the rays
  $[w, \gamma\cdot \eta_1)$ and $[w, \gamma\cdot \eta_2)$
  respectively. Thus, by \Cref{lem:weylchamber_basepoint}, for some
  uniform constant $K'$, depending on $\Theta''$, we have
  \begin{align*}
      d_o(\gamma x, \gamma y) &\le 2^{-K' d_X(o, w)} d_w(\gamma x, \gamma y)\\
      &\le 2^{-K' d_X(o, w)+CD}d_{\gamma\cdot o}(\gamma x, \gamma y)\\
      &\le 2^{-K'(L-CD)+CD}d_{o}(x,y),
  \end{align*}
  where the last inequality follows from the triangle inequality
  $d_X(o,w)\ge d_X(o,\gamma\cdot o)-d_X(\gamma\cdot o,w)\ge
  L-CD$. Then, if the constant $L$ satisfies
  $-K'(L-CD)+CD < -\log_2(E)$, the element $\gamma$ has the
  contracting property we desire.
\end{proof}

\section{Constructing proto-coders from cocompact actions}
\label{sec:constructing_coders}

In this section we actually set about constructing the proto-coders we
need to apply \Cref{thm:topological_criterion}. As before, fix a
locally compact Euclidean building $X$. From this point forward, we
also assume that $\Gamma < \Isom(X)$ is a discrete group acting
properly and cocompactly on $X$; thus $\Gamma$ acts on each flag space
associated to a face $\taumod$ of the model simplex $\sigmamod$. We
fix such a face $\taumod$, as well as a basepoint $o \in X$. We let
$\rho_0:\Gamma \to \Homeo(\flags)$ denote the induced (``standard'')
action of $\Gamma$ on $\flags$, although for the most part we still
leave this action implicit.

The main result of this section is the following:
\begin{proposition}
  \label{prop:proto_coder_construction}
  Let $\Theta$ be a compact subset of $\ost(\taumod)$, let $K \ge 1$,
  and let $A \ge 0$. There exists a $\rho_0$-adapted proto-coder
  $\mc{S}$ generating a contracting point coder $\mc{G}$ that
  satisfies both of the following properties:
  \begin{enumerate}
  \item\label{item:morse_geodesics_have_codes} There is a constant
    $B > 0$ satisfying the following. If $(x_n)$ is any
    $\Theta$-regular $(K,A)$-quasi-geodesic sequence in $X$, there is
    a $\mc{G}$-coding $\coding{c}$ with path sequence $(g_n)$ such
    that the Hausdorff distance between $(x_n)$ and $(g_no)$ is at
    most $B$.
  \item\label{item:codes_are_morse} There exists a compact set
    $\Theta' \subset \ost(\taumod)$ and constants
    $K' \ge 1, A', D' \ge 0$ satisfying the following. Suppose that
    $\coding{c}$ is a $\mc{G}$-coding of a point $\tau \in \flags$,
    with path sequence $(g_n)$. Then, the sequence $(g_no)$ is a
    $\Theta'$-regular $(K', A')$-quasi-geodesic, lying within a
    $D'$-neighborhood of $V(g_0o, \st_{\Theta'}(\tau))$.
  \end{enumerate}
\end{proposition} 

The two parts of the proposition complement each other: the first part
says that any uniformly regular quasi-geodesic in $X$ is uniformly
approximated by some coding, and the second part says that every
coding determines a uniformly regular quasi-geodesic (but possibly the
regularity and quasi-geodesic constants are worse than those from the
first part).

\subsection{Defining the proto-coder $\mathcal S$}

For the rest of the section, fix a compact subset
$\Theta \subset \ost(\taumod)$ and quasi-geodesic constants
$K \ge 1, A \ge 0$ as in the statement of
\Cref{prop:proto_coder_construction}. \emph{For convenience, $\Theta$ is chosen so that it contains a small ball, w.r.t.~Tits distance, around the barycenter of $\taumod$.} For the proof of both parts of the
proposition, we will need to apply the higher-rank Morse lemma
(\Cref{thm:morselemma}) as a key step, so we also fix an auxiliary
$\taumod$-convex compact set $\Theta' \subset \ost(\taumod)$ whose
interior contains $\Theta$.

Instead of constructing a single $\mc{S}$ satisfying the conclusions
of \Cref{prop:proto_coder_construction}, we will define a family of
proto-coders $\mc{S}(D, L, \eps)$ depending on real parameters
$L, D, \eps > 0$. Then we will show that for sufficiently large $L, D$
and sufficiently small $\eps > 0$, our construction gives a
proto-coder satisfying both of the desired properties.

\textbf{Step 1: define a pair of open coverings of $\flags$.} We will
employ the contraction estimate from the previous section
(\Cref{prop:expansivity}) to build a pair of nested open coverings of
$M$ consisting of ``expanding'' open sets. Precisely, for each flag
$z \in \flags$, consider the geodesic ray $[o, b_z)$ from $o$ to the
barycenter $b_z$ of $z$. Then, for any parameters $L, D > 0$, let
$I_z(D, L)$ be the set of elements $\gamma \in \Gamma$ satisfying both
of the following conditions:
\begin{enumerate}[label=(\Roman*)]
\item \label{item:thick_sphere_bound} $L \le d_X(o, \gamma o) \le 2L$,
\item \label{item:close_to_cone}
  $d_X(\gamma o, V(o, \st_{\Theta'}(z)) < D$.
\end{enumerate}
Although $I_z(D, L)$ depends on $L$ and $D$, from now on we will omit
this from the notation and just write $I_z(D, L) = I_z$.

Observe that $I_z$ is always finite due to the upper bound in
\ref{item:thick_sphere_bound} and proper discontinuity of the
$\Gamma$-action on $X$. Moreover, we have:
\begin{proposition}
  \label{prop:index_sets_nonempty}
  If $L \ge D \ge \diam(X/\Gamma)$, then $I_z$ is nonempty for every
  $z \in Z$.
\end{proposition}
\begin{proof}
  Let $u$ be any point on the ray
  $[o, b_z) \subseteq V(o, \st_{\Theta}(z))$ satisfying
  $d_X(o, u) = L + D$. Then any $\gamma$ satisfying
  $d_X(u, \gamma o) \le D$ verifies both \ref{item:thick_sphere_bound}
  and \ref{item:close_to_cone}, and such a $\gamma$ must exist as
  $\diam(X / \Gamma) \le D$.
\end{proof}

The above, combined with the contraction estimate from the previous
section, implies the following:
\begin{proposition}
  \label{prop:sets_contract}
  If $D > \diam(X/\Gamma)$, $L$ is sufficiently large (depending only
  on $D$) and $\eps$ is sufficiently small (depending only on $D$),
  then each element $\gamma$ in the nonempty set $I_z$ is
  $\frac12$-contracting on the ball $B(\gamma^{-1}z, \eps)$.
\end{proposition}
\begin{proof}
  Choose $L$ large enough and $\eps$ small enough as in our
  contraction estimate (\Cref{prop:expansivity}); in the hypotheses of
  the proposition, we take the expansivity constant $E$ to be $2$, $D$
  to be as above, and the compact set $\Theta$ to be our auxiliary
  compact set $\Theta' \subset \ost(\taumod)$. Then, if $\gamma$
  satisfies \ref{item:thick_sphere_bound} and \ref{item:close_to_cone}
  above, then it also satisfies conditions \ref{item:long_translation}
  and \ref{item:close_to_theta_star} in the hypotheses of
  \Cref{prop:expansivity} with this choice of parameters.
\end{proof}

Now, for each $\eps > 0$, and each $\gamma \in I_z$, define the pair
of subsets
\begin{align*}
  U(\gamma, z) &= \gamma B(\gamma^{-1}z, \eps/4),\\
  W(\gamma, z) &= \gamma B(\gamma^{-1}z, \eps/2).
\end{align*}
Once again, although $U(\gamma, z)$ and $W(\gamma, z)$ depend on a
choice of $\eps$, we omit this from the notation. Note that the
ultrametric property for the metric $d_o$ on $M$ implies that
$U(\gamma, z)$ and $W(\gamma, z)$ are both actually clopen subsets.

\textbf{Step 2: define $\mc{S} = \mc{S}(D, L, \eps)$ and verify that
  it is a proto-coder.} To define a proto-coder for the
$\rho_0$-action on $M$, we need to pick a pair of covers of $M$,
indexed by a common finite set $Z$, and then assign each element of
$z$ to a group element which is ``expanding'' on the corresponding
subsets of the cover. We have defined the sets $U(\gamma, z)$ and
$W(\gamma, z)$ so that $\gamma$ is (metrically) expanding on both of
them, so we just need to pick a finite index set
$Z \subset \Gamma \times M$. First, for each $z \in M$, consider the
clopen subset
\[
  O(z) = \bigcap_{\gamma \in I_z}U(\gamma, z).
\]
Note that $O(z)$ is nonempty as it contains $z$. The family of sets $\{O(z) : z \in M\}$ also gives a covering of $M$
by open sets, so there is a finite set $F \subset M$ such that
$\{O(z) : z \in F\}$ is also a covering. So, define
$Z \subset \Gamma \times M$ by
\[
  Z = \{(\gamma, z) : z \in F, \gamma \in I_z\}.
\]
Then $Z$ indexes the pair of open coverings
\[
  \{W(\gamma, z)\}_{(\gamma, z) \in Z}, \qquad \{U(\gamma,
  z)\}_{(\gamma, z) \in Z}.
\]
For each $(\gamma, z) \in Z$, define the ``expanding'' group element
$\alpha(\gamma, z)$ to be $\gamma$.

\begin{proposition}
  \label{prop:proto_coders_exist}
  For any sufficiently large $L, D$ and sufficiently small $\eps$, the
  index set $Z$, systems of sets $U(\gamma, z), W(\gamma, z)$, and
  group elements $\alpha(\gamma, z) = \gamma$ define a proto-coder
  $\mc{S}$ adapted to the standard action $\rho_0$, as in
  \Cref{defn:protocoder}.
\end{proposition}
\begin{proof}
  By construction, the sets $U(\gamma, z)$ and $W(\gamma, z)$ give a
  covering of $M$ indexed by $Z$, satisfying
  $U(\gamma, z) \subset W(\gamma, z)$. This means that $\mc{S}$ is a
  proto-coder, so we now just need to verify that it is adapted to the
  standard action $\rho_0$.

  Suppose that $(\gamma_1, z_1), (\gamma_2, z_2) \in Z$ satisfy
  $\alpha(\gamma_1, z_1)^{-1}U(\gamma_1,z_1))\cap
  U(\gamma_2,z_2)\neq \varnothing$. Then by definition of
  $U(\gamma_1,z_1)$ and $\alpha(\gamma_1,z_1)$,
  \[
    B\left(\gamma_1^{-1}z_1, \epsilon/4\right)\cap U(\gamma_2,z_2)
    =\gamma_1^{-1}U(\gamma_1,z_1))\cap U(\gamma_2,z_2)\neq
    \varnothing.
  \]
  Note that $\gamma_2$ is a $\frac 1 2$-contracting homeomorphism from
  a ball of radius $\frac \eps 4$ to $U(\gamma_2,z_2)$ and from a ball
  of radius $\frac \eps 2$ to $W(z_2)$. Thus $U(\gamma_2,z_2)$ and
  $W(\gamma_2,z_2)$ are contained in balls of radii $\frac \eps 8$ and
  $\frac \eps 4$, respectively. So, since $W(\gamma_2,z_2)$ contains
  at least one point in $B(\gamma_1^{-1}z_1, \epsilon/4)$, the
  ultrametric property implies
  \[
    \bar{W(\gamma_2,z_2)} = W(\gamma_2,z_2) \subset
    B(\gamma_1^{-1}z_1, \epsilon/4)\cap W(\gamma_2,z_2).
  \]
  But we also have
  \[
    B(\gamma_1^{-1}z_1, \epsilon/4)\cap W(\gamma_2,z_2) \subset
    B(\gamma_1^{-1}z_1, \epsilon/4)=\gamma_1^{-1}U(\gamma_1,z_1)
    \subset \gamma_1^{-1}W(\gamma_1, z_1),
  \]
  implying
  $\bar{W(\gamma_2, z_2)} \subset \gamma_1^{-1}W(\gamma_1, z_1)$, as
  required.
\end{proof}

\begin{remark}
  \label{rem:constant_dependence}
  We reiterate the order of dependence of the constants in the
  previous proposition, which ultimately comes from
  \Cref{prop:index_sets_nonempty} and \Cref{prop:sets_contract}: the
  conclusion of \Cref{prop:proto_coders_exist} holds as long as
  $D > \diam(X/\Gamma)$, $L$ is sufficiently large (depending only on
  $D$), and $\eps$ is sufficiently small (depending only on
  $D$). Later we will introduce dependence of $\eps$ on $L$, but we
  will be careful to avoid circularity; $L$ will never depend
  (directly or indirectly) on $\eps$.
\end{remark}

From this point forward, \emph{we will always assume that $L, D$ and
  $\eps$ are chosen so that $\mc{S}(D, L, \eps)$ actually defines a
  proto-coder as in the previous proposition.} So, we will increase
$L, D$ and decrease $\eps$ as needed throughout the following. We next
observe:
\begin{proposition}
  \label{prop:proto_coders_stable}
  The proto-coder $\mc{S}(D, L, \eps)$ defined above is stable in the
  sense of \Cref{defn:stable_system}.
\end{proposition}
\begin{proof}
  The statement follows immediately from the fact that each set
  $U(\gamma, z)$ is clopen and each $\alpha(\gamma, z)$ is a
  homeomorphism of $M$.
\end{proof}

Recall from \Cref{sec:generating_coders} that a proto-coder $\mc{S}$
adapted to $\rho_0$ determines a point coder $\mc{G}(\mc{S},
\rho_0)$. So, from \Cref{prop:proto_coders_exist}, we may define the
following.
\begin{definition}
  For any $D, L, \eps$ as above, let $\mc{G}(D, L, \eps)$ denote the
  point coder $\mc{G}(\mc{S}(D, L, \eps), \rho_0)$ generated by
  $\rho_0$ and the proto-coder $\mc{S}(D, L, \eps)$. We will refer to
  a $\mc{G}(D, L, \eps)$-coding as a \emph{$(D, L, \eps)$-coding}.
\end{definition}

Using the metric contraction of the action of each $\alpha(\gamma, z)$
on $\alpha(\gamma, z)^{-1}W(\gamma, z)$, we can prove:
\begin{proposition}
  \label{prop:coders_contracting}
  The point coder $\mc{G}(D, L, \eps)$ is $\rho_0$-contracting in the
  sense of \Cref{defn:contracting_coding}.
\end{proposition}
\begin{proof}
  We defined the sets $W(\gamma, z)$ and chose $L, \eps$ to ensure
  that if there is an edge in $\mc{G} = \mc{G}(D, L, \eps)$ from
  $(\gamma_1, z_1)$ to $(\gamma_2, z_2)$, labeled by
  $\alpha(\gamma_1, z_1) = \gamma_1$, then $\gamma_1$ is
  $\frac12$-contracting on $\gamma_1^{-1}W(\gamma_1, z_1)$. Therefore
  $\gamma_1$ is also $\frac12$-contracting on $W(\gamma_2,
  z_2)$. Since we have seen that the proto-coder is adapted to the
  action of $\Gamma$ on $\flags$, we also know that
  $\gamma_1W(\gamma_2, z_2)$ is a subset of $W(\gamma_1, z_1)$.

  Consider an infinite vertex path $((\gamma_k, z_k))_{k=1}^\infty$ in
  $\mc{G}$; by definition, the element labeling the edge from
  $(\gamma_k, z_k)$ to $(\gamma_{k+1}, z_{k+1})$ is
  $\gamma_k$. Applying the observation above inductively, it follows
  that the set
  \[
    N_k = \alpha_1 \cdots \alpha_kW(\gamma_{k+1}, z_{k+1})
  \]
  has diameter at most $2^{-k}\diam W(\gamma_{k+1}, z_{k+1})$, and
  that $N_{k+1} \subset N_k$ for all $k$. Since $M$ is compact, this
  shows that the intersection of the $N_k$ is a singleton.
\end{proof}

\subsubsection{Uniform separation} The next proposition tells us that,
if the parameter $\eps$ is chosen sufficiently small, then condition
\ref{item:pair_cocompactness} in \Cref{thm:topological_criterion} will
automatically hold.

\begin{proposition}
  \label{prop:separation_axiom_holds}
  If $\eps$ is sufficiently small, then for any distinct points
  $u, u' \in M$, there is a group element $\gamma \in \Gamma$ so that
  if $\gamma u \in U(\alpha, y)$ and $\gamma u' \in U(\beta, z)$, then
  \[
    W(\alpha, y) \cap W(\beta, z) = \emptyset.
  \]
\end{proposition}

Before we prove the proposition, we recall the following basic fact
about cocompact actions on Euclidean buildings. We provide a proof for
the reader's convenience.
\begin{lemma}
  \label{lem:pairs_cocompact}
  Let $M^{(2)}$ denote the space of pairs of distinct points in
  $M$. Then, there exists $\eps > 0$ so that for any pair
  $(u,v) \in M^{(2)}$, some element $\gamma \in \Gamma$ satisfies
  $d_o(\gamma u, \gamma v) > \eps$. Consequently, the action of
  $\Gamma$ on $M^{(2)}$ is cocompact.
\end{lemma}
\begin{proof}
  Let $u, v$ be a pair of distinct points in $M$. Then $u, v$ are
  contained in a common spherical apartment in $\vbdry X$ which is the
  boundary at infinity for a Euclidean apartment
  $\mathcal{A} \subset X$. Let $a$ be a basepoint in
  $\mathcal{A}$. Recall that we have assumed that the $\Gamma$-action
  on $X$ is cobounded, so (after translating $\mathcal{A}, a$, $u$,
  and $v$ by some element $\gamma \in \Gamma$) we may assume that the
  distance between $a$ and our fixed basepoint $o \in X$ is at most
  $\diam(X / \Gamma) < \infty$.

  Let $c^a_u:[0, \infty) \to X$ and $c^a_v:[0, \infty) \to X$ denote
  the unit-speed parameterizations of the geodesic rays from $a$ to
  the barycenters of $u, v$. The angle between these rays must be one
  of finitely many nonzero values, so there is some fixed $T > 0$
  (independent of $u$ and $v$) such that
  \[
    d_X(c_u^a(T), c_v^a(T)) \ge 2\diam(X / \Gamma) + 1.
  \]
  Now, if $c_u, c_v$ denote the unit-speed parameterizations of the
  geodesic rays from $o$ to the barycenters of $u,v$, convexity of the
  distance function implies that
  \[
    d_X(c_u(T), c_u^a(T)) \le d_X(c_u(0), c_u^a(0)) \le \diam(X /
    \Gamma).
  \]
  Similarly we have $d_X(c_v(T), c_v^a(T)) \le \diam(X / \Gamma)$, and
  therefore
  \[
    d_X(c_u(T), c_v(T)) \ge 1.
  \]
  By definition this implies $d_X(u,v) \ge \log_2(T)$.
\end{proof}

\begin{proof}[Proof of \Cref{prop:separation_axiom_holds}]
  Fix $\eps > 0$ as in the previous lemma, so that for any
  $(u,v) \in M^{(2)}$, we can find $\gamma \in \Gamma$ so that
  $d_o(\gamma u, \gamma v) > \eps$. Assume that the proto-coder
  $\mc{S}(D, L, \eps)$ is defined using an $\eps$ which is at least
  this small. Then, given distinct points $u, v \in M$, fix a $\gamma$
  as above, and suppose that $\gamma u \in U(\alpha, y)$ and
  $\gamma v \in U(\beta, z)$. Note that the sets $W(\alpha, y)$ and
  $W(\beta, z)$ both have diameter at most $\eps/2$, since they are
  the images of sets with diameter at most $\eps$ by
  $\frac12$-contracting elements. So, as
  \begin{align*}
    \gamma u \in U(\alpha, y) \subset W(\alpha, y),\\
    \gamma v \in U(\beta, z) \subset W(\beta, z),
  \end{align*}
  if there is some point $w$ in the intersection $W(\alpha, y) \cap
  W(\beta, z)$, then
  \[
    d_o(\gamma u, \gamma v) \le d_o(\gamma u, w) +
    d_o(\gamma v, w) \le \eps.
  \]
  This is a contradiction, so the intersection must be empty.
\end{proof}

\subsection{All regular quasi-geodesics are coded} We now turn to the
proof of the first condition in \Cref{prop:proto_coder_construction},
which can be viewed as a direct consequence of the following. Recall
that we have fixed a compact $\taumod$-convex subset
$\Theta \subset \ost(\taumod)$ and quasi-geodesic constants
$K \ge 1, A \ge 0$.
\begin{proposition}\label{prop:quasi-geodesic-is-coded}
  For any sufficiently large $D, L$ and any sufficiently small
  $\eps > 0$, there is a constant $B > 0$ satisfying the following. If
  $(x_n)_{n=0}^\infty$ is a $\Theta$-regular $(K,A)$-quasi-geodesic
  sequence in $X$, then there is a $(D, L, \eps)$-coding $\coding{c}$
  with path sequence $(g_n)_{n=0}^\infty$ such that the Hausdorff
  distance between $\{x_n\}_{n=0}^\infty$ and $\{g_no\}_{n=0}^\infty$
  is at most $B$.
\end{proposition}

Before proving \Cref{prop:quasi-geodesic-is-coded}, we need a pair of
general lemmas. The first says that any quasi-geodesic has a
``reparameterization'' which is coarsely proportional to its
arc-length.
\begin{lemma}\label{lem:reparametrize}
  Let $K \ge 1$ and $A \ge 0$. For every $L\ge K+A$, there exists
  $A' \ge 0$ such that every $\Theta$-regular $(K,A)$-quasi-geodesic
  $(x_n)$ in $X$ has a subsequence $(x_{n_k})$ satisfying the
  following properties:
  \begin{enumerate}
  \item $(x_{n_k})$ is a $\Theta$-regular quasi-geodesic,
  \item for every $k$, we have $L\le d_X(x_{n_k},x_{n_{k+1}})\le 2L$,
  \item the sequence $(x_n)$ and subsequence $(x_{n_k})$ are
    $A'$-uniformly close to each other.
  \end{enumerate}
\end{lemma}

\begin{proof}
  We construct our subsequence term-by-term, making sure we satisfy
  the second condition each time. Precisely, let $x_{n_1} =
  x_1$. Then, for all $k > 1$, inductively define $n_k$ to be the
  smallest index $n_k > n_{k-1}$ so that
  $d(x_{n_{k-1}}, x_{n_k}) \ge L$. Since $d_X(x_{n_k},x_{n_k-1})<L$
  and $(x_n)$ is $(K,A)$-quasi-geodesic, it follows from the triangle
  inequality that $d_X(x_{n_{k-1}},x_{n_k})<L+(K+A)\le 2L$.

  Now, since $(x_n)$ is a $(K,A)$-quasi-geodesic,
  $K^{-1}(n_{k+1}-n_k)-A \le d_X(x_{n_k},x_{n_{k+1}})\le
  K(n_{k+1}-n_k)+A$. It follows that $n_{k+1}-n_k\le K(2L+A)$. Thus,
  every element $x_n$ is within distance $A'=K^2(2L+A)+A$ of an
  element in the subsequence $(x_{n_k})$. This shows the third
  condition. The first condition also follows: any subsequence of a
  $\Theta$-regular sequence is automatically $\Theta$-regular, and
  since $j - k \le n_j - n_k$ for all $j \le k$ we have
  \[
    \frac{1}{K}|j - k| - A \le d_X(x_{n_j}, x_{n_k}) \le 2L|j - k|
  \]
  for all $k, j$.
\end{proof}
  
  The second lemma says that the cones over a pair of nearby stars
  locally agree when they share a basepoint. It is an immediate
  consequence of a result of Kapovich--Leeb--Porti.
  \begin{lemma}[see {\cite[Lem. 3.48]{KLP2018}}]
    \label{lem:taumod_cones_agree}
    For any $\ell > 0$, there exists $\eps > 0$ such that if
    $y, z \in M$ satisfy $d_o(y, z) < \eps$, then
    \[
      V(o, \st(y)) \cap B(o, \ell) = V(o, \st(z)) \cap B(o, \ell).
  \]
\end{lemma}

\begin{proof}[Proof of \cref{prop:quasi-geodesic-is-coded}]
  Fix a $\Theta$-regular $(K,A)$-quasi-geodesic sequence
  $(x_n)_{n=0}^\infty$ as in the proposition. Recall that we fixed a
  compact subset $\Theta' \subset \ost(\taumod)$ whose interior
  contains $\Theta$. As a consequence of the higher-rank Morse lemma
  (see \Cref{cor:improved_morse}), there is a constant $D > 0$
  (depending only on $K, A, \Theta, \Theta'$) and some fixed
  $z \in \flags$ so that for any $n < m$, the point $x_m$ lies within
  distance $D$ of the cone $V(x_n, \st_{\Theta'}(z))$.

  We can consider a sequence of the form $(g_no)_{n=0}^\infty$ for
  $g_n \in \Gamma$, such that the distance between $g_no$ and $x_n$ is
  at most $\diam(X/\Gamma)$. After increasing $D$ appropriately (and
  in particular ensuring that $D > \diam(X/\Gamma)$), it is also true
  that for any $n < m$, the point $g_mo$ lies in the $D$-neighborhood
  of $V(g_no, \st_{\Theta'}(z))$, because the Hausdorff distance
  between the Weyl cones $V(x_n, \st_{\Theta'}(z))$ and
  $V(g_no, \st_{\Theta'}(z))$ is at most
  $d_X(x_n, g_no) \le \diam(X/\Gamma)$. This is the parameter $D$ we
  use to define the proto-coder $\mc{S}(D, L, \eps)$.

  Next we fix a constant $L$, depending on $D$. There are two lower
  bounds $L$ needs to satisfy: first, it must be large enough for the
  conclusion of \Cref{prop:proto_coders_exist} to hold for any
  sufficiently small $\eps > 0$ (see
  \Cref{rem:constant_dependence}). Second, $L$ needs to be large
  enough for us to replace $(g_no)$ with a uniformly Hausdorff-close
  subsequence satisfying
  \begin{equation}
    \label{eq:sparse_subsequence}
    L \le d_X(g_no, g_{n+1}o) \le 2L
  \end{equation}
  for every $n$. Since $(g_no)$ is a quasi-geodesic sequence with
  quasi-geodesic constants depending only on $K,A$ and
  $\diam(X/\Gamma)$, this is possible due to \Cref{lem:reparametrize}.

  Our task is now to check that, if $\eps > 0$ is sufficiently small,
  the sequence $(g_n)_{n=0}^\infty$ is exactly the path sequence for
  some $(D, L, \eps)$-coding $\coding{c}$. Since
  $(g_no)_{n=0}^\infty$ has uniformly bounded Hausdorff distance to
  $(x_n)_{n=0}^\infty$, this is sufficient. We will assume
  throughout that $\eps > 0$ is chosen small enough to guarantee (via
  \Cref{prop:proto_coders_exist}) that $\mc{S}(D, L, \eps)$ is
  actually a proto-coder, and thus $(D, L, \eps)$-codings actually
  make sense; we will shrink $\eps$ as necessarily throughout the
  following argument.

  Precisely, defining $\gamma_n = g_{n-1}^{-1}g_n$ for all $n \ge 1$,
  we want to show that if $\eps < 0$ is sufficiently small, there is a
  sequence of elements $(z_n)_{n=1}^\infty$ in $M$ so that the
  sequence $(\gamma_n, z_n)_{n=1}^\infty$ specifies a vertex path in
  the graph $\mc{G} = \mc{G}(D, L, \eps)$, and therefore defines a
  $(D, L, \eps)$-coding starting at $g_0$ with path sequence
  $(g_n)_{n=0}^\infty$.

  We start by defining a candidate sequence $z_n$. Recall that we
  previously found a finite subset $F \subset M$ so that the
  $F$-indexed collection of intersections
  \[
    \left\{\bigcap_{\gamma \in I_z}U(\gamma, z) : z \in F\right\}
  \]
  gives an open covering of $M$. In particular, for each $n \ge 1$,
  there is some $z_n \in F$ so that
  \[
    g_{n-1}^{-1}z \in \bigcap_{\gamma \in I_{z_n}} U(\gamma, z_n).
  \]
  We need to show both that $\gamma_n \in I_{z_n}$ for all $n \ge 1$
  (meaning $(\gamma_n, z_n)$ is in the vertex set $Z$ of $\mc{G}$ for
  all $n \ge 1$) and that there is an edge from $(\gamma_n, z_n)$ to
  $(\gamma_{n+1}, z_{n+1})$ for all $n \ge 1$.

  Note that for any $\gamma \in I_{z_n}$, the set $U(\gamma, z)$ is
  contained in a $d_o$-ball of radius $\eps/8$, since $\rho(\gamma)$
  is $\frac12$-contracting on
  $\rho(\gamma)^{-1}U(\gamma, z_n) \subset B(\rho(\gamma^{-1})z_n,
  \eps/4)$. So by the ultrametric property we have
  $d_o(z_n, g_{n-1}^{-1}z) < \eps/8$. Thus, by
  \Cref{lem:taumod_cones_agree}, if $\eps > 0$ is small enough, we can
  guarantee that
  \begin{equation}
    \label{eq:cones_agree_short}
    V(o, \st_{\Theta'}(z_n)) \cap B(o, 2L + D) = V(o,
    \st_{\Theta'}(g_{n-1}^{-1}z)) \cap B(o, 2L + D).
  \end{equation}
  
  Now, as $g_no$ lies within distance $D$ of
  $V(g_{n-1}o, \st_{\Theta'}(z))$, the point
  $g_{n-1}^{-1}g_no = \gamma_no$ lies within distance $D$ of
  $V(o, \st_{\Theta'}(g_{n-1}^{-1}z))$. From
  \eqref{eq:sparse_subsequence} we have
  \[
    L \le d_X(o, \gamma_no) \le 2L.
  \]
  The upper bound, together with \eqref{eq:cones_agree_short},
  therefore tells us that $\gamma_no$ must lie within distance $D$ of
  $V(o, \st_{\Theta'}(z_n))$. This---combined with the lower bound on
  $d_X(o, \gamma_no)$---is precisely what is needed to ensure that
  $\gamma_n \in I_{z_n}$ (recall that $I_{z_n}$ is defined by the
  conditions \ref{item:thick_sphere_bound}, \ref{item:close_to_cone}
  above).

  Finally we note that there is an edge from $(\gamma_n, z_n)$ to
  $(\gamma_{n+1}, z_{n+1})$: this occurs exactly when the intersection
  \[
    \rho(\gamma_n)^{-1}U(\gamma_n, z_n) \cap U(\gamma_{n+1}, z_{n+1})
  \]
  is nonempty, and by construction, the point
  $g_n^{-1}z = \gamma_n^{-1}g_{n-1}^{-1}z$ lies in both sets.
\end{proof}

\begin{remark}
  The proof of \Cref{prop:quasi-geodesic-is-coded} introduces a
  dependence of $\eps$ on $L$, mentioned previously in
  \Cref{rem:constant_dependence}; we will see another such dependence
  in the sequel.
\end{remark}

\subsection{Every coding gives a regular quasi-geodesic}

Now we wish to prove the second condition in
\Cref{prop:proto_coder_construction}. The desired statement is a
consequence of the following more technical proposition. Recall that,
as part of the definition of the family of proto-coders
$\mc{S}(D, L, \eps)$, we fixed an auxiliary $\taumod$-convex compact
subset $\Theta' \subset \ost(\taumod)$ whose interior contains
$\Theta$.
\begin{proposition}
  \label{prop:codes_are_morse}
  Fix any $D > 0$ and a compact subset
  $\Theta'' \subset \ost(\taumod)$ whose interior contains
  $\Theta'$. For any sufficiently large $L$ and any sufficiently small
  $\eps > 0$, there exist $K' \ge 1$, $A', D' \ge 0$ satisfying the
  following: if $\coding{c}$ is a $(D, L, \eps)$-coding of a point
  $z \in M$ with path sequence $(g_n)_{n=0}^\infty$, then the sequence
  $(g_no)_{n=0}^\infty$ is a $\Theta''$-regular
  $(K', A')$-quasi-geodesic. Moreover, this quasi-geodesic lies within
  a $D'$-neighborhood of $V(g_0o, \st(z))$.
\end{proposition}

The ``moreover'' part of the proposition is almost a direct
consequence of the first part combined with the higher-rank Morse
lemma (\Cref{thm:morselemma}): if the first part of the proposition
holds, then the Morse lemma says that the sequence $(g_no)$ stays
close to the cone over \emph{some} star $\st(y)$. However, one still
needs to check that, if $(g_n)$ is the path sequence for a coding of
$z$, then $(g_no)$ actually approaches the specific star $\st(z)$.

The first step in the proof of \Cref{prop:codes_are_morse} is the
following consequence of uniform contraction:
\begin{lemma}
  \label{lem:coding_dist_lowerbound}
  There exists a constant $B$ satisfying the following. Given any
  $R > 0$, if $L$ is sufficiently large (depending on $R, D$) and
  $\eps$ is sufficiently small (depending on $L, R, D$), and
  $\coding{c}$ is a $(D, L, \eps)$-coding with path sequence
  $(g_n)_{n=0}^\infty$, then for all $n \ge 0$ we have
  \[
    d_X(g_0o, g_no) \ge Rn - B.
  \]
\end{lemma}
Since the distance between consecutive points in a coding is uniformly
bounded by $2L$, this lemma immediately implies that
$(D, L, \eps)$-codings determine quasi-geodesic sequences.
\begin{proof}
  Without loss of generality, assume $g_0 = \id$. Let $E \ge 2$ be
  arbitrarily chosen. We can choose some $L = L(E)$ at least as large
  as the value given by \Cref{prop:expansivity}, with our given value
  of $E$ and $\Theta$, $D$ chosen as in the proof of
  \Cref{prop:coders_contracting}; then, if
  $(\gamma_n, z_n)_{n=1}^\infty$ denotes the sequence of vertices in
  $\mc{G}(D, L, \eps)$ determining a coding $\coding{c}$, our
  definition of $\mc{S}(D, L, \eps)$ ensures that each $\gamma_n$ is
  an $E^{-1}$-contracting map from $W(\gamma_{n+1}, z_{n+1})$ to
  $W(\gamma_n, z_n)$. By the nesting property of vertex sets in the
  proto-coder, it follows inductively that for all $n \ge 1$, $g_n$ is
  a $E^{-n}$-contracting map from $W(\gamma_{n+1}, z_{n+1})$ to
  $W(\gamma_1, z_1)$.

  Let $S$ be a finite generating set for $\Gamma$. Each $s \in S$
  defines a bi-Lipschitz homeomorphism on $M$, meaning that there is a
  constant $0 < C < 1$ so that $d_o(sx, sy) \ge Cd_o(x,y)$ for all
  $x, y \in M$ and all $s \in S$. Inductively, if some element
  $h \in \Gamma$ has length at most $n$ (with respect to the word
  metric on $\Gamma$ induced by $S$), we have
  $d_o(hx, hy) \ge C^nd_o(x,y)$. Consequently, the word-length of
  $g_n$ is at least $-n\log(E)/\log(C)$. Finally, since $\Gamma$ acts
  properly and $D$-coboundedly on $X$, the Milnor-Schwartz lemma tells
  us that the orbit map into $X$ based at $o$ is a quasi-isometric
  embedding, which gives the desired lower bound if we increase $L$
  (hence $E$) appropriately.
\end{proof}

Our next task is to show that $(D, L, \eps)$-codings give uniformly
regular sequences. We start with an a priori weaker statement: we
prove that the orbits of codings in $X$ can experience at most
uniformly linear drift away from the cone over a star. 
\begin{lemma}
  \label{lem:uniform_linear_drift}
  For any sufficiently small $\eps > 0$ (depending on $D, L$), if
  $\coding{c}$ is a $(D, L, \eps)$-coding of $z \in M$ with path
  sequence $(g_n)_{n=0}^\infty$, then for all $n \ge 0$ we have
  \[
    d_X(g_no, V(g_0o, \st_{\Theta'}(z))) \le nD.
  \]
\end{lemma}
Note that this lemma also introduces a dependence of $\eps$ on $L$
(see \Cref{rem:constant_dependence}).
\begin{proof}
  Again, without loss of generality, take $g_0 = \id$. Let
  $(\gamma_n, z_n)_{n=1}^\infty$ denote the sequence of vertices in
  $Z$ determining the coding $\coding{c}$; since the coding starts
  from the identity, we have $g_n = \gamma_1 \cdots \gamma_n$. By
  definition, since $\coding{c}$ codes $z$, we have
  \[
    z \in \bigcap_{n=1}^\infty g_{n-1} W(\gamma_n, z_n),
  \]
  and thus for every $n \ge 1$, we have
  $g_{n-1}^{-1}z \in W(\gamma_n, z_n)$. Now, for every $n$, we know
  that $g_{n-1}^{-1}z \in W(\gamma_n, z_n)$, meaning that
  $d_o(g_{n-1}^{-1}z, z_n) \le \eps/2$. So, by
  \Cref{lem:taumod_cones_agree}, we can choose $\eps$ small enough so
  that
  \begin{equation}
    \label{eq:coding_cones_agree}
    V(o, \st_{\Theta'}(z_n)) \cap B(o, 2L + D) = V(o,
    \st_{\Theta'}(g_{n-1}^{-1}z)) \cap B(o, 2L + D)
  \end{equation}
  holds for all $n$. Our definition of the proto-coder
  $\mc{S}(D, L, \eps)$ ensures that $\gamma_n o$ lies in the
  $D$-neighborhood of the cone $V(o, \st_{\Theta'}(z_n))$, so for each
  $n$ fix a point $p_n \in V(o, \st_{\Theta'}(z_n))$ with
  $d_X(\gamma_n o, p_n) \le D$.

  We now prove the desired inequality via induction on $n$. Since
  $g_0 = \id$, the statement is immediate when $n = 0$, so suppose
  that $n \ge 1$. Inductively, there is a point
  $q \in V(o, \st_{\Theta}(z))$ such that
  \[
    d_X(q, g_{n-1}o) \le (n-1)D.
  \]
  Now, since $d_X(o, p_n) \le d_X(o, \gamma_no) + D \le 2L + D$, the
  equality in \eqref{eq:coding_cones_agree} tells us that there is a
  point $\xi_n \in \st_{\Theta}(z)$ so that the geodesic ray
  $[o, g_{n-1}^{-1}\xi_n)$ passes through $p_n$; equivalently, the ray
  $[g_{n-1}o, \xi_n)$ passes through $g_{n-1}p_n$.  By convexity, the
  Hausdorff distance between the rays $[g_{n-1}o, \xi_n)$ and
  $[q, \xi_n)$ is at most $d_X(q, g_{n-1}o) \le (n-1)D$. This means
  that the distance between $g_no = g_{n-1}\gamma_no$ and $[q, \xi_n)$
  is at most
  \begin{align*}
    d_X(g_no, g_{n-1}p_n) + d_X(g_{n-1}p_n, [q, \xi_n))
    &=d_X(\gamma_no, p_n) + d_X(g_{n-1}p_n, [q, \xi_n))\\
    &\le D + (n-1)D = nD.
  \end{align*}
  Using the nesting property of cones over $\taumod$-convex stars
  (\Cref{lem:theta_cone_nesting}), since
  $q \in V(o, \st_{\Theta'}(z))$ and $\xi_n \in \st_{\Theta'}(z)$ we
  have $[q, \xi_n) \subset V(o, \st_{\Theta'}(z))$, hence
  $d_X(g_no, V(o, \st_{\Theta'}(z)) \le nD$ as required.
\end{proof}

\begin{proof}[Proof of \Cref{prop:codes_are_morse}]
  Let
  $\delta = \frac{1}{4}\dt(\Theta', \sigmamod \minus \ost(\taumod))$;
  in particular this implies that the closed $3\delta$-neighborhood of
  $\Theta'$ is a compact subset of $\ost(\taumod)$. Let $\Theta''$ be
  the $\delta$-neighborhood of $\Theta'$, and let $B > 0$ be the
  constant from \Cref{lem:coding_dist_lowerbound}. Choose $R > 0$
  large enough so that
  \[
    \sin^{-1}\left(\frac{Dn}{Rn - B}\right) < \delta
  \]
  for all $n \ge 1$. Then fix $L$ (depending on $R, D$) large enough
  to ensure that the conclusion of \Cref{lem:coding_dist_lowerbound}
  holds, and finally fix $\eps > 0$ small enough (depending on
  $R, D, L$) so that the conclusion of \Cref{lem:uniform_linear_drift}
  holds.

  Fix a $(D, L, \eps)$-coding $\coding{c}$ for a point $z \in M$, with
  path sequence $(g_n)_{n=0}^\infty$; as before, we can assume without
  loss of generality that $g_0 = \id$. Let $x_n = g_no$. We will show
  that $\theta(\vec{ox_n}) \in \Theta''$ for all $n$. By
  \Cref{lem:uniform_linear_drift}, there is a point
  $p_n \in V(o, \st_{\Theta'}(z))$ with
  \[
    d_X(p_n, x_n) < nD.
  \]
  Consequently the CAT(0) comparison angle
  $\angle_{\bar{o}}(\bar{p_n}, \bar{x_n})$ is at most
  \[
    \sin^{-1}\left(\frac{d_X(o, x_n)}{d_X(x_n, p_n)}\right) \le
    \sin^{-1}\left(\frac{Dn}{Rn - B}\right) < \delta.
  \]
  It follows that the angle between the directions
  $\theta(\vec{ox_n})$ and $\theta(\vec{op_n})$ is also at most
  $\delta$, ensuring that $\theta(\vec{ox_n}) \in \Theta''$.

  To show that $(x_n)_{n=0}^\infty$ is $\Theta''$-regular, we need to
  further show that $\theta(\vec{x_ix_j}) \in \Theta''$ for all
  $i < j$, but this follows from the above: if the vertex sequence for
  the coding $\coding{c}$ is $(\gamma_n, z_n)_{n=1}^\infty$, then for
  any $i \ge 0$ the sequence $(\gamma_{i+n}, z_{i+n})_{n=1}^\infty$
  also determines a $(D, L, \eps)$-coding, with initial point at the
  identity. The argument above (applied to this coding) shows
  precisely that for any $j > i$, the geodesic segment from $o$ to
  $\gamma_{i+1} \cdots \gamma_jo = g_i^{-1}g_jo$ is
  $\Theta''$-regular, and therefore so is the geodesic segment from
  $g_io$ to $g_jo$.

  It remains to show the ``moreover'' part of the proposition,
  i.e. that the sequence $(x_n)_{n=0}^\infty$ lies in the
  $D'$-neighborhood of the cone $V(o, \st(z))$, for some uniform
  $D'$. We have just shown that $(x_n)_{n=0}^\infty$ specifies a
  $\Theta''$-regular $(K', A')$-quasi-geodesic, so the higher rank
  Morse lemma (\Cref{thm:morselemma}) provides some uniform $D'$ and
  some point $z' \in M$ so that $x_n$ lies within distance $D'$ of
  $V(o, \st(z'))$. So, we only need to check that $z' = z$. The idea
  is to inspect the convergence of the sequence $x_n$ to a point in
  $\partial_\infty X$ to show that the angle between two points in
  $\st_{\Theta''}(z)$ and $\st(z')$ is small, and then apply the
  separated stars lemma (\Cref{lem:separated_stars}).

  Fix a sequence of points $p_n' \in V(o, \st(z'))$ with
  $d_X(x_n, p_n') < D'$. Since each cone $V(o, \st(z))$ and
  $V(o, \st(z'))$ is a union of geodesic rays based at $o$, for each
  $n$ we can find points $\xi_n \in \st(z)$ and $\xi_n' \in \st(z')$
  so that the geodesic ray $[o, \xi_n)$ passes through $p_n$, and
  $[o, \xi_n')$ passes through $p_n'$. Up to extracting a subsequence,
  $\xi_n$ converges to a point $\xi \in \st(z)$ and $\xi_n'$ converges
  to some $\xi' \in \st(z')$. In fact, since
  $p_n \in V(o, \st_{\Theta''}(z))$, we know that
  $\xi_n \in \st_{\Theta''}(z)$, hence $\xi \in \st_{\Theta''}(z)$ by
  the compactness of $\Theta''$-stars (see \Cref{lem:stars_closed}).

  Note that, if $y, y'$ are points on the rays $[o, \xi), [o, \xi')$,
  the CAT(0) comparison angle $\angle_{\bar{o}}(\bar{y}, \bar{y}')$
  converges to the angular distance $\dt(\xi, \xi')$ as $y \to \xi$
  and $y' \to \xi'$ (see e.g. \cite[Proposition II.9.8]{BH1999}). So,
  we may choose $y, y'$ so that this CAT(0) comparison angle is at
  least $\dt(\xi, \xi') - \delta/8$. We have points
  $y_n \in [o, \xi_n)$ and $y_n' \in [o, \xi_n')$ respectively
  converging to $y, y'$. See \Cref{fig:limit_rays}.

  \begin{figure}[ht]
    \centering
    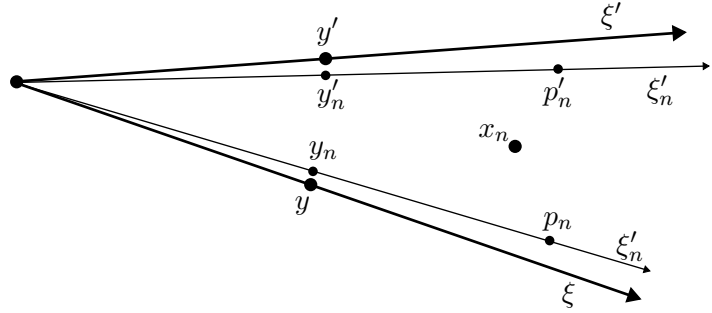
    \caption{The rays $[o,\xi_n'), [o,\xi_n)$ converge to rays
      $[o,\xi), [o,\xi')$. By bounding $\angle_o(x_n, \xi_n)$ and
      $\angle_o(x_n, \xi_n')$, we bound $\angle_o(\xi, \xi')$.}
    \label{fig:limit_rays}
  \end{figure}

  Since $p_n$ and $p_n'$ both leave every bounded subset of $X$, as
  long as $n$ is large enough, $y_n$ lies on the segment $[o, p_n]$
  and $y_n'$ lies on the segment $[o, p_n']$, which implies that for
  large $n$ the CAT(0) comparison angle
  $\angle_{\bar{o}}(\bar{p_n}, \bar{p_n}')$ satisfies
  \begin{equation}
    \label{eq:cat0_comparison}
    \angle_{\bar{o}}(\bar{p_n}, \bar{p_n}') \ge \dt(\xi, \xi') -
    \delta/4.
  \end{equation}
  We have chosen $p_n'$ so that $d_X(p_n', x_n) \le D'$, which means
  that $d_X(o, p_n') \ge Rn - B - D'$ and thus the comparison angle
  $\angle_{\bar{o}}(\bar{p_n}', \bar{x_n})$ satisfies
  \[
    \angle_{\bar{o}}(\bar{p_n}', \bar{x_n}) \le
    \sin^{-1}\left(\frac{D'}{Rn - B - D'}\right).
  \]
  In particular for large enough $n$ we have
  $\angle_{\bar{o}}(\bar{p_n}', \bar{x_n}) \le \delta/4$. We also
  observed above that
  $\angle_{\bar{o}}(\bar{p_n}, \bar{x_n}) \le \delta$, yielding
  $\angle_{\bar{o}}(\bar{p_n}, \bar{p_n}') \le 5\delta/4$. Combining
  this with \eqref{eq:cat0_comparison} we obtain
  $\angle(\xi, \xi') \le 3\delta/2$.

  As $\Theta''$ was defined to be the $\delta$-neighborhood of
  $\Theta'$, our choice of $\delta$ ensures that
  $\dt(\Theta'', \sigmamod \minus \ost(\taumod)) \ge 2\delta$. Since
  $\xi$ belongs to $\st_{\Theta''}(z)$ and $\xi'$ belongs to
  $\st(z')$, \Cref{lem:separated_stars} implies that $z = z'$.
\end{proof}

\section{Interpolation}
\label{sec:interpolation}

We now complete the proof of the main theorem by applying
\Cref{thm:topological_criterion} to the proto-coders constructed in
the previous section. We have already done the work needed to verify
that our proto-coders satisfy three out of the four conditions
required by \Cref{thm:topological_criterion}; the next proposition is
the key to proving that the last condition also holds. In the
statement below, when we say that a pair of sequences $(x_n)$ and
$(y_m)$ in $X$ are \emph{$D$-close infinitely often} for some constant
$D > 0$, we mean that there are infinitely many pairs of indices
$(m,n)$ so that $d_X(x_m, y_n) \le D$.

\begin{proposition}[Interpolation]
  \label{prop:interpolating_quasigeodesics}
  For any compact subset $\Theta \subset \ost(\taumod)$ and constants
  $K \ge 1, A \ge 0$, there is a compact set
  $\Theta' \subset \ost(\taumod)$ and constants
  $K' \ge 1, A', D \ge 0$ satisfying the following. For any pair of
  $\Theta$-regular $(K,A)$-quasi-geodesic sequences $(x_n), (y_n)$
  with the same $\taumod$-limit $\tau \in \flags$, there is an
  ``interpolating'' $\Theta'$-regular $(K', A')$-quasi-geodesic
  sequence $(z_n)$ which is $D$-close infinitely often to both $(x_n)$
  and $(y_n)$.
\end{proposition}

As a preliminary step, we prove the following.

\begin{lemma}\label{lem:enter-time}
  Let $\tau$ be a $\taumod$–simplex in $\vbdry X$, let $\Theta$ be a
  compact $\taumod$–convex subset of $\ost(\taumod)$, and assume that
  $\Theta$ contains the barycenter of $\taumod$. Then, for any subset
  $\Theta' \subset \ost(\taumod)$ containing a neighborhood of
  $\Theta$, there exists a constant $C > 0$, depending on $\Theta'$,
  such that for every $x, y \in X$,
  \[
    \{z\in V(y,\st_{\Theta}(\tau)):d_X(y,z)>Cd_X(x,y)\}\subset
    V(x,\st_{\Theta'}(\tau)).
  \]
\end{lemma}
\begin{proof}
  Fix $\delta > 0$ so that a $\delta$-neighborhood of $\Theta$ is
  contained in $\Theta'$, and let $\eta\in \vbdry X$ be the barycenter
  of $\tau$. Let $p$ be the point on $[y,\eta)$ where $[y,\eta)$ first
  enters $V(x,\st_{\Theta'}(\tau))$. By
  \Cref{cor:geodesics_enter_cones},
  \[
    d_X(y,p)<
    \left(\frac{1}{\sin\delta}+\frac{1}{\sin^2\delta}\right)d_X(x,y).
  \]
  Let $z\in V(y,\st_\Theta(\tau))$. There exists
  $\xi\in \st_{\Theta}(\tau)$ such that $z\in [y,\xi)$. Note that
  $V(y,\st(\tau))$ is the union of sectors based at $y$ containing the
  $\taumod$-sector $V(y,\tau)$. In particular, there is a sector $S$
  containing both rays $[y,\eta)$ and $[y,\xi)$. Since
  $p\in [y,\eta)$, the ray $[p,\xi)$ is contained in $S$ and is
  parallel to the ray $[y,\xi)$. On the other hand, $p$ is contained
  in $V(x,\st_{\Theta'}(\tau))$, so the ray $[p,\xi)$ is also
  contained in $V(x,\st_{\Theta'}(\tau))$. By the convexity of the
  cone $V(x,\st_{\Theta'}(\tau))$, any ray in the sector $S$ starting
  from $p\in V(x,\st_{\Theta'}(\tau))$ making an angle at most
  $\delta$ with the ray $[p,\xi)$ is contained in
  $V(x,\st_{\Theta'}(\tau))$.

  Let $q$ be the last point along the geodesic $[y, \xi)$ such that
  $\angle_q(y,p) = \delta$. Note that if $z\in [y,\xi)$ satisfies
  $d_X(y,z)>d_X(y,q)$ then we have $\angle_z(p,q) \le \delta$, hence
  $\angle_p(\xi, z) \le \delta$, and therefore
  $z \in V(x, \st_{\Theta'}(\tau))$. On the other hand, by the
  Euclidean law of sines,
  \[
    d_X(y,q)=d_X(y,p)\frac{\sin\angle_p(y,q)}{\sin\delta}\le
    \frac{d_X(y,p)}{\sin\delta},
  \]
  and therefore
  \[
    d_X(y, q) \le \left(\frac{1}{\sin^2\delta} +
      \frac{1}{\sin^3\delta}\right)d_X(x,y).
  \]
  Thus setting
  $C = \left(\frac{1}{\sin^2\delta} + \frac{1}{\sin^3\delta}\right)$
  we see that if $d_X(x,z) > Cd_X(x,y)$, we get
  $z \in V(x, \st_{\Theta'}(\tau))$.
\end{proof}

\begin{proof}[Proof of \cref{prop:interpolating_quasigeodesics}]
  By the Morse lemma (and \Cref{cor:improved_morse}), we may enlarge
  $\Theta, K, A$ slightly and replace $(x_n)$ and $(y_n)$ with
  uniformly close sequences satisfying
  $x_n \in V(x_0, \st_{\Theta}(\tau))$ and
  $y_n \in V(y_0, \st_{\Theta}(\tau))$ for all $n \ge 0$. By
  translating by an isometry of $X$ and perturbing the sequences by
  another small amount, we can also assume that $x_0 = o$. We may also
  assume that $\Theta$ contains the barycenter of $\taumod$.

  Let $\Theta'$ be a $\taumod$-convex compact subset of
  $\ost(\taumod)$ containing a neighborhood of $\Theta$ (so in
  particular $\Theta'$ contains a neighborhood of the barycenter of
  $\taumod$). We shall inductively construct a $\Theta'$-regular
  sequence $(w_k)$, whose terms alternate between $(x_n)$ and $(y_n)$,
  with uniform gaps between consecutive terms. More precisely, we will
  find a sequence of indices $(n_k)$ so that the sequence $(w_k)$
  defined by
  \[
    w_k =
    \begin{cases}
      x_{n_k}, &k \textrm{ even}\\
      y_{n_k}, &k \textrm{ odd}
    \end{cases}
  \]
  satisfies $w_{k+1} \in V(w_k, \st_{\Theta'}(\tau))$ and
  $d_X(w_k, w_{k+1}) \ge 1$ for all $k$.

  To construct $(n_k)$, set $n_0 = 0$, and suppose inductively that
  $n_k$ (hence $w_k$) has already been defined for some $k \ge 0$;
  without loss of generality assume $k$ is even, so
  $w_k = x_{n_k} \in V(x_0, \st_{\Theta'}(\tau))$. Then define
  $n_{k+1}$ so that $d_X(w_k, y_{n_{k+1}}) \ge 1$, and
  \[
    n_{k+1} \ge \min\{m>n_k: y_n\in V(w_{n_k},\st_{\Theta'}(\tau))
    \textrm{ for all } n\ge m\}.
  \]
  \Cref{lem:enter-time} implies that the minimum above actually
  exists. Indeed, since $\Theta'$ contains a neighborhood of $\Theta$
  and of the barycenter of $\taumod$, there is a constant $C > 0$ so
  that $y_n \in V(w_{n_k}, \st_{\Theta'}(\tau))$ whenever
  $d_X(y_n, y_0) > Cd_X(w_{n_k}, y_0)$, and both distances
  $d_X(y_n, y_0)$ and $d_X(y_n, w_k)$ tend to infinity since $(y_n)$
  is a quasi-geodesic sequence.

  We now construct the sequence $(z_n)$ by ``filling in the gaps''
  between consecutive terms in $(w_k)$ using geodesic segments. That
  is, we let $(z_n)$ be a sequence in $X$, starting with $w_0$ and
  containing $(w_k)$ as a subsequence $(z_{n_k})$, such that the terms
  in $(z_n)$ appearing between $(z_{n_{k-1}})$ and $(z_{n_k})$ are
  evenly spaced along the geodesic $[z_{n_{k-1}}, z_{n_k}]$. Since
  $d_X(w_k, w_{k+1}) \ge 1$, we can do this while arranging that
  $1 \le d_X(z_n, z_{n+1}) \le 2$ for all $n$. See
  \Cref{fig:interpolation}.

  \begin{figure}[ht]
    \centering
\begingroup%
  \makeatletter%
  \providecommand\color[2][]{%
    \errmessage{(Inkscape) Color is used for the text in Inkscape, but the package 'color.sty' is not loaded}%
    \renewcommand\color[2][]{}%
  }%
  \providecommand\transparent[1]{%
    \errmessage{(Inkscape) Transparency is used (non-zero) for the text in Inkscape, but the package 'transparent.sty' is not loaded}%
    \renewcommand\transparent[1]{}%
  }%
  \providecommand\rotatebox[2]{#2}%
  \newcommand*\fsize{\dimexpr\f@size pt\relax}%
  \newcommand*\lineheight[1]{\fontsize{\fsize}{#1\fsize}\selectfont}%
  \ifx\svgwidth\undefined%
    \setlength{\unitlength}{318.40398834bp}%
    \ifx\svgscale\undefined%
      \relax%
    \else%
      \setlength{\unitlength}{\unitlength * \real{\svgscale}}%
    \fi%
  \else%
    \setlength{\unitlength}{\svgwidth}%
  \fi%
  \global\let\svgwidth\undefined%
  \global\let\svgscale\undefined%
  \makeatother%
  \begin{picture}(1,0.53181163)%
    \lineheight{1}%
    \setlength\tabcolsep{0pt}%
    \put(0,0){\includegraphics[width=\unitlength,page=1]{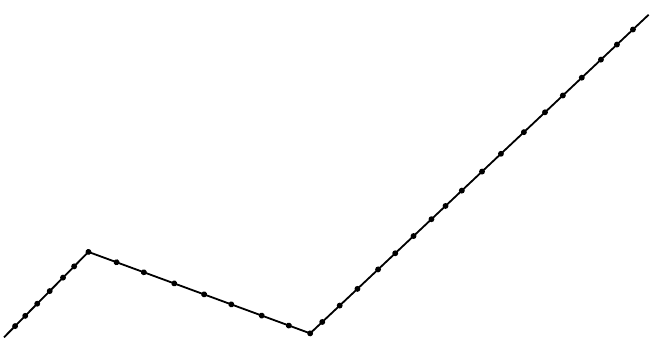}}%
    \put(0.79124106,0.04522984){\color[rgb]{0,0,0}\makebox(0,0)[lt]{\lineheight{1.25}\smash{\begin{tabular}[t]{l}$y_m$\end{tabular}}}}%
    \put(0.4967325,0.39873578){\color[rgb]{0,0,0}\makebox(0,0)[lt]{\lineheight{1.25}\smash{\begin{tabular}[t]{l}$x_n$\end{tabular}}}}%
    \put(0.44844955,0.00560205){\color[rgb]{0,0,0}\makebox(0,0)[lt]{\lineheight{1.25}\smash{\begin{tabular}[t]{l}$w_k$\end{tabular}}}}%
    \put(0.88464968,0.51678695){\color[rgb]{0,0,0}\makebox(0,0)[lt]{\lineheight{1.25}\smash{\begin{tabular}[t]{l}$w_{k+1}$\end{tabular}}}}%
    \put(0.69878213,0.22873341){\color[rgb]{0,0,0}\makebox(0,0)[lt]{\lineheight{1.25}\smash{\begin{tabular}[t]{l}$z_i$\end{tabular}}}}%
    \put(0,0){\includegraphics[width=\unitlength,page=2]{interpolation.pdf}}%
  \end{picture}%
\endgroup%

    \caption{The ``interpolating'' sequence $(z_i)$ between the two
      sequences $(x_n)$ and $(y_m)$.}
    \label{fig:interpolation}
  \end{figure}

  We claim that the sequence $(z_n)$ is $\Theta'$-regular. To see
  this, first note that for any $k$, and any
  $n_{k-1} \le n < m \le n_k$, we have
  $z_m \in V(z_n, \st_{\Theta'}(\tau))$ by convexity of
  $\Theta'$-cones, hence
  $V(z_m, \st_{\Theta'}(\tau)) \subset V(z_n, \st_{\Theta'}(\tau))$ by
  the nesting property for $\Theta'$-cones
  (\Cref{lem:theta_cone_nesting}). Now, for arbitrary indices $n < m$,
  fix indices $k \le j$ such that $n_k \le m < n_{k+1}$ and
  $n_j \le m < n_{j+1}$. If $k = j$, then we get
  $V(z_m, \st_{\Theta'}(\tau)) \subset V(z_n, \st_{\Theta'}(\tau))$.
  Otherwise, we have $k < j$, hence $k+1 \le j$. By applying the
  nesting property for $\Theta'$-cones inductively, we get
  \[
    V(z_{n_j}, \st_{\Theta'}(\tau)) = V(w_j, \st_{\Theta'}(\tau))
    \subseteq V(w_{k+1}, \st_{\Theta'}(\tau)) = V(z_{n_{k+1}},
    \st_{\Theta'}(\tau)),
  \]
  hence
  \[
    V(z_m, \st_{\Theta'}(\tau)) \subseteq V(z_{n_j},
    \st_{\Theta'}(\tau)) \subseteq V(z_{n_{k+1}}, \st_{\Theta'}(\tau))
    \subseteq V(z_n, \st_{\Theta'}(\tau)).
  \]
  This proves that $\theta(\vec{z_nz_m}) \in \Theta'$, so $(z_n)$ is
  $\Theta'$-regular.

  It remains to prove that $(z_n)$ is uniformly quasi-geodesic. The
  arguments above show that for any $n < m$, we have
  $z_m \in V(z_n, \st_{\Theta'}(\tau))$, and thus we have a nested
  sequence of convex cones
  \[
    V(z_0, \st_{\Theta'}(\tau)) \supset V(z_1, \st_{\Theta'}(\tau))
    \supset \ldots
  \]
  It follows that the piecewise-geodesic path made up of the union of
  geodesic segments $[z_i, z_{i+1}]$ is \emph{$\Theta'$-longitudinal}
  in the sense of \cite[Def 3.10]{KLP2018}, so due to
  \cite[Lem. 3.11]{KLP2018}, it is uniformly bi-Lipschitz (with
  constant depending only on $\Theta'$) if it is parameterized by arc
  length. Since the distance between consecutive points in $(z_n)$ is
  bounded between 1 and 2, this shows that $(z_n)$ is a quasi-geodesic
  sequence.
\end{proof}

\begin{proof}[Proof of \Cref{thm:main_theorem}]
  We construct a pair of proto-coders $\mc{S}, \mc{S}'$ satisfying the
  hypotheses of \Cref{thm:topological_criterion}. Each proto-coder
  will come from a separate application of
  \Cref{prop:proto_coder_construction}.

  First we construct the proto-coder $\mc{S}$, which is chosen to
  satisfy condition \ref{item:codes_are_morse} in
  \Cref{prop:proto_coder_construction}. This proto-coder is adapted to
  the (standard) $\rho_0$-action of $\Gamma$ on $\flags$ and generates
  a point coder $\mc{G}$. The proposition says that there is a compact
  subset $\Theta \subset \ost(\taumod)$ and constants
  $K \ge 1, A \ge 0$ so that if $(g_n)$ is the path sequence for
  some $\mc{G}$-coding of a flag $\tau \in \flags$, the sequence
  $(g_no)$ is a $\Theta$-regular $(K,A)$-quasi-geodesic, lying in a
  uniform neighborhood of $V(o, \st(\tau))$. For this step, the first
  condition in \Cref{prop:proto_coder_construction} is irrelevant, so
  we can make an arbitrary choice of parameters in the hypotheses of
  the proposition.

  Then, from \Cref{prop:interpolating_quasigeodesics}, we can find a
  compact subset $\Theta' \subset \ost(\taumod)$ and constants
  $K' \ge 1, A' \ge 0$ so that if $(g_n)_{n=0}^\infty$ and
  $(h_m)_{m=0}^\infty$ are the path sequences associated to any two
  $\mc{G}$-codings of the same point $\tau \in \flags$, there is a
  $\Theta'$-regular $(K', A')$-quasi-geodesic sequence $(x_\ell)$ with
  $\taumod$-limit $\tau$ that is uniformly close to $(g_no)$
  infinitely often and uniformly close to $(h_mo)$ infinitely
  often. Using these parameters, we construct a new $\rho_0$-adapted
  proto-coder $\mc{S}'$ satisfying condition
  \ref{item:morse_geodesics_have_codes} in
  \Cref{prop:proto_coder_construction}. That is, if $\mc{G}'$ is the
  point coder generated by $\mc{S}'$, every $\Theta'$-regular
  $(K',A')$-quasi-geodesic sequence in $X$ has uniformly bounded
  Hausdorff distance from the orbit of a basepoint in $X$ under the
  path sequence from some $\mc{G}'$-coding.

  We need to check that the proto-coders $\mc{S}, \mc{S}'$ satisfy all
  of the conditions \ref{item:stable}-\ref{item:meandering} in
  \Cref{thm:topological_criterion}. Conditions \ref{item:stable},
  \ref{item:contracting}, and \ref{item:pair_cocompactness} follow
  immediately from
  \Cref{prop:proto_coders_stable,prop:coders_contracting,prop:separation_axiom_holds},
  respectively, so it remains to show that there is a finite set
  $F \subset \Gamma$ so that $\mc{S}$ and $\mc{S}'$ satisfy condition
  \ref{item:meandering}. Let $(g_n)$ and $(h_m)$ be the path sequences
  associated to a pair of $\mc{G}$-codings of the same point
  $\tau \in \flags$. As observed above, this implies that there is a
  $\Theta'$-regular $(K', A')$-quasi-geodesic sequence $(x_\ell)$ with
  $\taumod$-limit $\tau$ which is uniformly close to $(g_no)$
  infinitely often and uniformly close to $(h_mo)$ infinitely often.

  Our construction of $\mc{S}'$ tells us that there is an
  $\mc{S}'$-coding of a flag $\tau' \in \flags$, with path sequence
  $(f_k)_{k=0}^\infty$, so that the Hausdorff distance between
  $(f_ko)$ and $(x_\ell)$ is at most a constant $D > 0$ (independent
  of all of the codings). The second condition in
  \Cref{prop:proto_coder_construction} says that the $\taumod$-limit
  of $(f_ko)$ is $\tau'$, and \Cref{cor:theta_cones_separated} implies
  that $(x_\ell)$ also has $\taumod$-limit $\tau'$. Thus by uniqueness
  of $\taumod$-limits we have $\tau = \tau'$.

  Now, from the construction, there are infinitely many indices $k$ so
  that $f_ko$ lies within uniform distance of some point of the form
  $g_no$. Since $\Gamma$ acts properly on $X$, we conclude that there
  is a fixed finite set $F$ so that $(f_k)$ and $(g_n)$ are
  $F$-close. Arguing symmetrically, the same is true for $(f_k)$ and
  $(h_m)$, and this concludes the proof of property
  \ref{item:meandering}.
\end{proof}

\section{Necessity of semi-conjugacy}\label{sec:blowsupexample}

In this section, we provide an example showing that in general it is
not possible to replace the semi-conjugacies in
\Cref{thm:main_theorem} with actual (topological) conjugacies. For
concreteness, we work in rank 2, although the construction can be
generalized to any rank $\ge 2$. For a rank-one example, see
\cite[Example 1.4]{MannManningWeisman22}.

\subsection{The construction}

Let $M = \mathbb{P}(\Q_p^3)$ be the $p$-adic projective plane, let
$G = \PGL_3(\Q_p)$, and let $\Gamma < G$ be a uniform lattice in $G$. In particular, $\Gamma$ is finitely generated (hence countable).

For each point $z \in M$, let $M_z$ denote the canonical blowup of $M$
at $z$. The blowup $M_z$ comes with a projection map
$\pi_z:M_z \to M$, such that $\pi_z^{-1}(y)$ is a singleton for all
$y \ne z$ and $\pi_z^{-1}(z)$ is identified with the projectivization
of the tangent space to $M$ at $z$ (a copy of the $p$-adic projective
line $\mathbb{P}(\Q_p^2)$). The blowup operation is natural in the
sense that any $g \in G$ induces a homeomorphism $M_z \to M_{gz}$
(also denoted $g$) satisfying $\pi_{gz} \circ g = g \circ \pi_z$.

Now, fix a point $z \in M$ with trivial $\Gamma$-stabilizer, so the
$\Gamma$-orbit $Z = \Gamma z$ is in bijection with $\Gamma$. To see
that such a point exists, note that for each nontrivial element
$g \in G$, the complement of the union of eigenspaces of $g$ is an
open dense subset of $M$, and then apply countability of $\Gamma$ and
the Baire category theorem.

Consider the countable family of maps
$\mathcal{P} = \{M_z \to M : z \in Z\}$. This family can be organized
into an (undirected) inverse system, and taking the inverse limit in
the topological category yields a space $\hat{M}$ equipped with a
projection $\hat{\pi}:\hat{M} \to M$.  The $\Gamma$-action on $Z$ defines a
$\Gamma$-action of $\Gamma$ on $\mathcal{P}$, and by the naturality of
the projection maps $\pi_z$, this induces an action
$\hat{\rho}:\Gamma \to \Homeo(\hat{M})$. If we let
$\rho_0:\Gamma \to \Homeo(M)$ denote the action on $M$ coming from the
inclusion of $\Gamma$ into $G$, then $\hat{\rho}$ satisfies
$\hat{\pi} \circ \hat{\rho}(\gamma) = \rho_0(\gamma) \circ \hat{\pi}$
for all $\gamma \in \Gamma$.

Observe that for each $z \in Z$, the blowup $M_z$ is homeomorphic to a
Cantor set. One way to see this is to let $L(z)$ denote the space of
two-dimensional linear subspaces in $\Q_p^3$ containing $z$, and
realize $M_z$ as the subspace
\[
  \{(w, \ell) \in M \times L(z) : w \subset \ell\}.
\]
Thus $M_z$ is identified with a nonempty perfect closed subspace of
$M \times L(z)$. Since $M = \mathbb{P}(\Q_p^3)$ and
$L(z) \simeq \mathbb{P}(\Q_p^2)$ are each homeomorphic to a Cantor
set, it follows that $M_z$ is nonempty, perfect, compact, metrizable,
and totally disconnected, so it is also homeomorphic to a Cantor set.

This in turn implies that the limit $\hat{M}$ is also homeomorphic to
a Cantor set, since $\hat{M}$ can be realized as the subspace
\[
  \left\{(x_z) \in \prod_{z \in Z} M_z : \pi_z(x_z) = \pi_y(x_y) \quad
  \forall y,z \in Z\right\}.
\]
This is a nonempty perfect closed subset of a countable product of
Cantor sets, so it is homeomorphic to a Cantor set.

Now equip $M$ with a metric $d$ compatible with its topology. We have:
\begin{proposition}
  \label{prop:blowup_homeo}
  For every $\eps > 0$, there exists a homeomorphism $h:\hat{M} \to M$
  satisfying \[d(h(x), \hat{\pi}(x)) < \eps,\] for all $x \in \hat{M}$.
\end{proposition}
\begin{proof}
  This is a consequence of the fact that, if $f:C \to C$ is any
  continuous surjective map of Cantor sets, there exists a
  homeomorphism $h:C \to C$ satisfying $d(h(x), f(x)) < \eps$ for any
  $x \in C$. To see that this holds, partition $C$ into a finite
  collection of pairwise disjoint nonempty clopen subsets $D_i$, each
  with diameter at most $\eps$. Each $D_i$ is then homeomorphic to a
  Cantor set. Moreover, each preimage $f^{-1}(D_i)$ is a nonempty
  clopen subset of $C$, so it is also homeomorphic to a Cantor set,
  and we can define $h$ on each $f^{-1}(D_i)$ to be an arbitrary
  homeomorphism to $D_i$.
\end{proof}

For each $\eps > 0$, let $h_\eps:\hat{M} \to M$ be a homeomorphism as
in the previous proposition. We can use $h_\eps$ to conjugate the
action $\hat{\rho}:\Gamma \to \Homeo(\hat{M})$ to an action
$\rho_\eps:\Gamma \to \Homeo(M)$.
\begin{proposition}
  \label{prop:nearby_conjugate}
  For any fixed $\gamma \in \Gamma$, the homeomorphism
  $\rho_\eps(\gamma)$ converges to $\rho_0(\gamma)$ as $\eps \to 0$.
\end{proposition}
\begin{proof}
  Fix $x \in M$, and let $y = \hat{\pi}(h_\eps^{-1}(x))$. Then we have
  $d(x,y) = d(h_\eps(h_\eps^{-1}(x)), \hat{\pi}(h_\eps^{-1}(x))) <
  \eps$. By the definition of $\rho_\eps$ and
  $(\hat{\rho}, \rho_0)$-equivariance of $\hat{\pi}$, we also have
  \[
    d(\rho_\eps(\gamma)x, \rho_0(\gamma)y) = d((h_\eps \circ
    \hat{\rho}(\gamma) \circ h_\eps^{-1})(x), (\hat{\pi} \circ
    \hat{\rho}(\gamma) \circ h_\eps^{-1})(x)) < \eps.
  \]
  Since $d(x,y) < \eps$, uniform continuity of $\rho_0(\gamma)$
  implies that $d(\rho_0(\gamma)x, \rho_0(\gamma)y) \to 0$ uniformly
  as $\eps \to 0$. Thus
  $d(\rho_\eps(\gamma)x, \rho_0(\gamma)x) \le d(\rho_\eps(\gamma)x,
  \rho_0(\gamma)y) + d(\rho_0(\gamma)y, \rho_0(\gamma)x))$ tends
  uniformly to zero.
\end{proof}

Since $\Gamma$ is finitely generated, this tells us that the
$\rho_\eps$-action lies in an arbitrarily small neighborhood of
$\rho_0$ in $\Hom(\Gamma, \Homeo(M))$. Finally we show:
\begin{proposition}
  The action $\rho_0$ is not conjugate to $\rho_\eps$ for any
  $\eps > 0$.
\end{proposition}
\begin{proof}
  Since $\rho_\eps$ is conjugate to $\hat{\rho}$, it suffices to show
  that $\hat{\rho}$ is not conjugate to $\rho_0$. We will show that
  the $\hat{\rho}$-action is not cocompact on the set of pairs of
  distinct points in $\hat{M}$ (see \Cref{lem:pairs_cocompact}).

  First, define a metric $\hat{d}$ on $\hat{M}$ as follows. Fix an
  enumeration $\gamma_0, \gamma_1, \gamma_2, \ldots$ of $\Gamma$, with
  $\gamma_0 = \id$. Since $z$ has trivial $\Gamma$-stabilizer,
  defining $z_n = \gamma_nz$ gives an enumeration of the orbit
  $Z = \Gamma z$. Fix an arbitrary metric $d_0$ inducing the correct
  topology on the blowup $M_z = M_{z_0}$, and for each $n$, let $d_n$
  be the metric on $M_{z_n}$ obtained by pushing $d_0$ forward by the
  homeomorphism $\gamma_n:M_{z_0} \to M_{z_n}$. Then define $\hat{d}$
  on the product $\prod_{n=0}^\infty M_{z_n}$ by
  \[
    \hat{d}\left((x_n), (y_n)\right) = \sum_{n=0}^\infty
    2^{-n}d_n(x_n, y_n).
  \]
  The metric $\hat{d}$ induces the product topology. As we noted
  previously, $\hat{M}$ embeds into this product, so $\hat{d}$
  restricts to a metric on $\hat{M}$ inducing the correct topology.

  Now, for each $r > 0$, let $\hat{M}^{(2)}_r$ denote the set of pairs
  of distinct points in $\hat{M}$ whose $\hat{d}$-distance is at most
  $r$. It suffices to prove that for every $r > 0$, there exists a
  pair of points whose $\Gamma$-orbit under the $\hat{\rho}$-action is
  entirely contained in $\hat{M}^{(2)}_r$. Note that, with respect to
  the metric $\hat{d}$, the diameters of the fibers
  $\hat{\pi}^{-1}(z_n)$ tend to zero as $n \to \infty$. So, fix
  $N > 0$ so that, for all $n > N$, the diameter of the fiber
  $\hat{\pi}^{-1}(z_n)$ is less than $r$. Then, pick a pair of
  distinct points $x,y$ in the fiber $\hat{\pi}^{-1}(z_0)$ so that for
  each element $\gamma_n$ with $0 \le n \le N$, we have
  $\hat{d}(\hat{\rho}(\gamma_n)x, \hat{\rho}(\gamma_n)y) < r$. By
  construction, for every
  $\gamma \in \Gamma \minus \{\gamma_0, \ldots, \gamma_N\}$, both
  points in the pair $\hat{\rho}(\gamma) x, \hat{\rho}(\gamma) y$ lie
  in some fiber $\hat{\pi}^{-1}(z_n)$ for $n > N$, and therefore also
  lie within distance $r$ of each other. Thus the $\Gamma$-orbit of
  $(x,y)$ lies in $\hat{M}^{(2)}_r$ and we are done.
\end{proof}

\begin{remark}
  The arguments in \Cref{prop:blowup_homeo} and
  \Cref{prop:nearby_conjugate} above actually show the following
  general statement: if $\sigma, \rho$ are two actions of the same
  finitely generated group on a Cantor set, and $\sigma$ is
  semi-conjugate to $\rho$, then $\sigma$ is conjugate to an
  arbitrarily small perturbation of $\rho$.
\end{remark}

\bibliography{research.bib}{}
\bibliographystyle{alpha}

\end{document}